\documentclass{amsart}
\usepackage{color,graphicx,enumerate,amssymb}
\usepackage{hyperref}
\usepackage{epsfig,wrapfig}
\usepackage{pxfonts}
\usepackage{graphicx}
\usepackage{eucal}
\usepackage[utf8]{inputenc}
\usepackage{ulem}

\usepackage{amsmath}
\usepackage{enumitem}
\newtheorem{theorem}{Theorem}[section] 
\newtheorem{lemma}[theorem]{Lemma}
\newtheorem{corollary}[theorem]{Corollary}
\newtheorem{definition}{Definition}[section]
\newtheorem{proposition}[theorem]{Proposition}
\newtheorem{remark}[theorem]{Remark}

\newcommand{\supp}{\operatorname{supp}}

\newcommand{\cK}{{\mathcal K}}
\newcommand{\bR}{{\mathbb R}}
\newcommand{\bN}{{\mathbb N}}
\newcommand\norm[1]{\Arrowvert {#1}\Arrowvert}

\def\avint{\mathop{\,\rlap{$\diagup$}\!\!\int}\nolimits}

\title[]{Regularity in the two-phase Bernoulli problem for the $p$-Laplace operator}

\author[M. Bayrami]{Masoud Bayrami}
\address{Department of Mathematical Sciences, Sharif University of Technology, Tehran, Iran}
\email{masoud.bayrami1990@sharif.edu}

\author[M. Fotouhi]{Morteza Fotouhi}
\address{Department of Mathematical Sciences, Sharif University of Technology, Tehran, Iran}
\email{fotouhi@sharif.edu}

\date{\today}

\begin{document}

\begin{abstract}
We show that any minimizer of the well-known ACF  functional (for the $p$-Laplacian) is  a viscosity solution. 
This allows us to   establish a uniform flatness  decay  at the two-phase free boundary points to improve the flatness, which boils down to $C^{1,\eta}$ regularity of  the flat part of the  free boundary. 
This result, in turn, is used to prove the Lipschitz regularity of  minimizers by a dichotomy argument.  
It is noteworthy that the analysis of branch points is also included. 
\end{abstract}


\thanks{M. Bayrami and M. Fotouhi was supported by Iran National Science Foundation (INSF) under project No. 4001885.}

\maketitle

\section{Introduction and main result}
We study the problem of minimizing the following two-phase functional
\begin{equation*}
J_{\mathrm{TP}}(v,D) :=\int_{D} |\nabla v|^p + (p-1)\lambda_+^p \chi_{\{v>0\}} + (p-1)\lambda_-^p \chi_{\{v<0\}} \, dx,
\qquad v \in \cK,
\end{equation*}
where $D$ is a bounded and smooth domain in $\mathbb{R}^n$, $\chi_A$ is the characteristic function of the set $A$,
 $1<p<\infty$,  and $\lambda_\pm > 0$ are given constants.
The class of admissible functions $\cK$, consists of all functions $v \in g + W_0^{1,p}(D)$,  where  $g \in W^{1,p}(D)$.

Any  minimizer $u$ satisfies, in a certain weak sense, the following system of equations
\begin{equation}
\label{OVERDETERMINED}
\begin{cases}
\Delta_p u := \mathrm{div}( |\nabla u|^{p-2} \nabla u ) = 0, & \qquad  \text{in} \quad \,\,\, \Omega_u^+ \cup \Omega_u^-, \\
|\nabla u^+|^p-|\nabla u^-|^p= \lambda_+^p-\lambda_-^p,  & \qquad  \text{on} \quad \left(\partial \Omega^+_u \cap \partial\Omega^-_u \right) \cap D, \\
 |\nabla u^+| \geq \lambda_{+},\,\,\,  |\nabla u^-| \geq \lambda_{-},  &  \qquad \text{on} \quad \left(\partial \Omega^+_u \cap \partial\Omega^-_u \right) \cap D, \\
|\nabla u^+| = \lambda_{+}, & \qquad  \text{on} \quad \left(\partial\Omega^+_u \setminus \partial\Omega^-_u \right) \cap D, \\
|\nabla u^-| = \lambda_{-}, & \qquad  \text{on} \quad \left(\partial\Omega^-_u \setminus \partial\Omega^+_u \right) \cap D,
\end{cases}
\end{equation}
where  $\Omega^\pm _u = \{x \in D \, : \, \pm u(x) > 0\}$,  $u^\pm :=\max\{\pm u,0\}$,
and $\Delta_p u = \mathrm{div}( |\nabla u|^{p-2} \nabla u)$ is the $p$-Laplace operator; see Lemma \ref{L2}. 

These types of problems are known as Bernoulli-type free boundary problems which appear in various models of fluid mechanics or heat conduction (see e.g. \cite{MR682265, MR733897, MR740956, MR772122, karakhanyan2021regularity, ferrari2022regularity}).
For the admissible functions in $\cK^+:=\{v \in \cK \, : \, v \geq 0\}$, the analogous one-phase functional and the corresponding overdetermined problem  called the one-phase Bernoulli problem, was first studied in \cite{MR618549} for the case $p=2$, and then in \cite{MR732100} for the two-phase problem. Also, the case of uniformly elliptic quasilinear equations in the one-phase case has been treated in \cite{MR752578}. 
The difficulty of the problem \eqref{OVERDETERMINED} is that the governing operator, $\Delta_p u = \mathrm{div}( |\nabla u|^{p-2} \nabla u )$, is not uniformly elliptic. 
Obviously, close to   regular free boundary points one expects that   $|\nabla u| > 0$  implying uniform ellipticity of the $p$-Laplacian.
However, without such a regularity assumption, it is difficult to prove non-degeneracy up to the free boundary. In \cite{MR2133664}, the authors circumvent this issue by simultaneously showing the non-degeneracy of the gradient and the regularity of the free boundary.

Here below we list  terminologies and definitions that are  frequently used in this paper: 
\begin{itemize}
\item
A function $u : D \to \mathbb{R}$ is said to be a \textit{minimizer} of $J_{\mathrm{TP}}$ \textit{in} $D$ if and only if 
$$ J_{\mathrm{TP}}(u,D)\leq J_{\mathrm{TP}}(v,D), $$
for all $v\in \cK$.

\item
 $F(u):=\left(\partial \Omega^+_u \cup \partial \Omega^-_u \right) \cap D$, denotes the free boundary of the minimizer $u$.
\item
The set $\Gamma_{\mathrm{TP}}:= \partial \Omega^+_u \cap \partial \Omega^-_u \cap D$ is the two-phase points of the free boundary $F(u)$.
\item
The boundary of positive and negative phases, i.e. $\partial \Omega^{\pm}_u \cap D$ can be decomposed as 
$$ \partial \Omega^{\pm}_u \cap D = \Gamma^{\pm}_{\mathrm{OP}} \cup \Gamma_{\mathrm{TP}},$$
where $\Gamma^{+}_{\mathrm{OP}}:=\left(\partial \Omega^{+}_u \setminus \partial \Omega^{-}_u \right) \cap D$ and $\Gamma^{-}_{\mathrm{OP}}:=\left(\partial \Omega^{-}_u \setminus \partial \Omega^{+}_u \right) \cap D$ are the one-phase parts of $F(u)$. 
\item
We will say that $x_0 \in \Gamma_{\mathrm{TP}}$ is an \textit{interior} two-phase point and will denote it by $x_0 \in \Gamma^{\mathrm{int}}_{\mathrm{TP}}$, if
$$ | B_r(x_0) \cap \{u=0\} | = 0, \qquad \text{for some} \quad r>0. $$
\item
We will say that $x_0 \in \Gamma_{\mathrm{TP}}$ is a \textit{branching} point and will denote it by $x_0 \in \Gamma^{\mathrm{br}}_{\mathrm{TP}}$, if
$$ | B_r(x_0) \cap \{u=0\} | > 0, \qquad \text{for every} \quad r>0. $$

\item
We denote by $H_{\alpha,\textbf{e}}$ the following one-dimensional function
\begin{equation*}
H_{\alpha,\textbf{e}}(x)=\alpha \left( x \cdot \textbf{e}\right)^+-\beta \left( x \cdot \textbf{e} \right)^-, 
\end{equation*}
with a unit vector $\textbf{e}\in \mathbb{S}^{n-1}$ and the constants $\alpha$ and $\beta$ satisfying the conditions 
\begin{equation}
\label{E0-Con}
\alpha \geq \lambda_+, \qquad \beta \geq \lambda_-, \qquad \alpha^p-\beta^p= \lambda^p_{+}-\lambda^p_{-}.
\end{equation}
$H_{\alpha,\textbf{e}}$ is the so-called {\it two-plane} solution to \eqref{OVERDETERMINED}.
\end{itemize}

Our goal is to study the regularity of the free boundary $F(u)=\left(\partial \Omega^+_u \cup \partial \Omega^-_u \right) \cap D$, for minimizers of $J_{\mathrm{TP}}$ in $D$, around the two-phase points. 
More precisely, we expect that  in a suitable neighborhood of the two-phase points, the sets $\Omega^+_u$ and $\Omega^-_u$ are two $C^{1,\eta}$-regular domains touching along the closed set of two-phase points $\Gamma_{\mathrm{TP}}$. 
For the special case $p=2$, this result has been recently obtained in \cite{MR4285137}, by invoking the linearization technique  and we will closely follow this technique in order to generalize this result to any $1<p<\infty$.

As is usual for problems of this type, prior to applying any method to determine the regularity of the free boundary, the Lipschitz continuity of the minimizers across the free boundary is required. 
Our partial result for the regularity of the free boundary, however, gives us the Lipschitz regularity of the solution as well.
We first show $C^{1,\eta}$-regularity of the free boundary with a flatness assumption in the following theorem.

\begin{theorem}[Flatness implies $C^{1,\eta}$]
\label{T1}
Let $u:D \to \mathbb{R}$ be a minimizer of $J_{\mathrm{TP}}$ in $D$ and $0 \in \Gamma_{\mathrm{TP}}$.
For any positive constants $\Lambda_0$ and $\Lambda_1$, there exists a constant $\bar \epsilon=\bar\epsilon(n,p,\Lambda_0,\Lambda_1)$ such that if 
\begin{equation}\label{T1-Eq1}
\norm{u-H_{\alpha,\textbf{e}}}_{L^\infty(B_1)} \le \bar\epsilon,
\end{equation}
for some  $\textbf{e}\in \mathbb{S}^{n-1}$ and $\max(\Lambda_0, \lambda_+)\le  \alpha \le \Lambda_1$, then  $\partial \Omega^{\pm}_u \cap B_{r_0}$ are $C^{1,\eta}$ graphs for some $r_0>0$ and for any $\eta\in(0, \frac13)$.
\end{theorem}

We need to remark that the critical flatness $\bar\epsilon$ in \eqref{T1-Eq1}, to obtain the regularity does not depend on $\lambda_\pm$. 
Indeed, as long as we are close enough to a two-plane solution with coefficient $\alpha\in[\Lambda_0,\Lambda_1]$, 
we obtain  the regularity of the free boundary.
 This result in turn is crucial to obtain the Lipschitz regularity of minimizers in the following theorem.

\begin{theorem}[Lipschitz regularity]
\label{P1}
Let $u:D \to \mathbb{R}$ be a minimizer of $J_{\mathrm{TP}}$ in $D$. Then $u$ is locally Lipschitz continuous; $u \in C_{\mathrm{loc}}^{0,1}(D)$.
\end{theorem}

\section{Basic properties of minimizers}
In this section, we gather some basic properties of  minimizers of $J_{\mathrm{TP}}$. Even though the continuity of the minimizer will be mentioned later (see proposition \ref{P0}), let us use it for the first part of the proof of the next Theorem, i.e. when we are assuming that the sets $\Omega_u^+$, and $\Omega_u^-$ are open.

\begin{theorem}[Existence]
\label{EoM}
If the admissible set $\cK$ is nonempty, then there exists a minimizer $u$ of $J_{\mathrm{TP}}$ over $\cK$. 
Moreover, every minimizer satisfies 
\[\begin{cases}
\Delta_p u = 0, &\qquad \mathrm{ in } \quad \Omega^+_u\cup\Omega^-_u, \\
\Delta_p u^{\pm} \ge0,&\qquad \mathrm{ in } \quad D,\\
\|u\|_{L^{\infty}(D)} \leq \|g\|_{L^{\infty}(D)}.
\end{cases}\] 
\end{theorem}
\begin{proof}
The existence of a bounded minimizer $u$ of the functional $J_{\mathrm{TP}}$ can be easily established using the semi-continuity of the $p$-Dirichlet energy and the weak convergence in $W^{1,p}$, and can be obtained in the standard way. See e.g. \cite{MR732100,  MR2126143} for the details. Also, notice that by  comparison of $u$ and $u + t \varphi$, where $\varphi$ is a suitable smooth that $\supp\varphi\subset\Omega_u^+\cup\Omega_u^-$, it is easy to find that $\Delta_{p} u= 0$ in $\Omega_u^+\cup\Omega_u^-$ in the sense of distributions. 

To prove that $u^+$ are $p$-subharmonic, we first note that 
since $\Delta_{p} u=0$ in $\Omega_u^+$, we may choose $\epsilon_k \to 0$ such that $ \{u=\epsilon_k\}$ to be a $C^1$ manifold by the Sard's Theorem, resulting in $\frac{-\nabla u}{|\nabla u|}$ to be the outer normal vector on $\partial\{u>\epsilon_k\}$.
Now take $0 \leq \varphi \in C^{\infty}_c(D)$, the integration by parts implies that
$$ 
\begin{aligned}
\int_{D} |\nabla u^+|^{p-2} \nabla u^+ \cdot \nabla \varphi \, dx & = \int_{\{u>0\}} |\nabla u^+|^{p-2} \nabla u^+ \cdot \nabla \varphi \, dx\\
& = \lim_{\epsilon_k \to 0} \int_{\{u> \epsilon_k\}} |\nabla u^+|^{p-2} \nabla u^+ \cdot \nabla \varphi \, dx \\
& = \lim_{\epsilon_k \to 0}   \int_{\{u=\epsilon_k\}} |\nabla u^+|^{p-2} \left( \nabla u^+ \cdot \frac{-\nabla u}{|\nabla u|} \right) \varphi \, dx  \\
& \qquad \quad \, - \int_{\{u> \epsilon_k\}} \Delta_{p} u \,  \varphi \, dx  \\
& = - \lim_{\epsilon_k \to 0} \int_{\{u=\epsilon_k\}} |\nabla u^+|^{p-1} \varphi \, dx \leq 0.
\end{aligned}
$$
The proof of $\Delta_p u^- \ge0$ is the same.
Finally, the last estimate 
$$ \|u\|_{L^{\infty}(D)} \leq \|g\|_{L^{\infty}(D)}, $$ 
is the consequence of the $p$-subharmonicity of $u^{\pm}$ in $D$.
\end{proof}

In the following proposition  we show the non-degeneracy property for the minimizers. It reveals the fact that each of the two phases $\Omega_u^+$  and $\Omega_u^-$ are optimal with respect to one-sided inwards perturbations. 
The proof is the same as the proof of non-degeneracy for one-phase problems; see \cite[Lemma 4.2]{MR2133664}.
We postpone the proof to Appendix \ref{appendix:nondegeneracy}.

\begin{proposition}[Non-degeneracy]
\label{P1.5}
Let $D \subset \mathbb{R}^n$ be an open set, and $u$ be a minimizer of $J_{\mathrm{TP}}$. Then,
$u$ is non-degenerate; i.e. there is a constant $C = C(n, \lambda_{\pm},p) > 0$ such that
$$ \avint_{B_r\left(x_0\right)} \left(u^{\pm}\right)^p \, dx \geq C r^p, $$
 for every $x_0 \in \overline{\Omega_u^{\pm}} \cap D$  and every $0<r<\operatorname{dist}\left(x_0, \partial D\right)$.
\end{proposition}

The next proposition concerns the Lipschitz regularity of the minimizers around the one-phase points.

\begin{proposition}[Lipschitz regularity at one-phase points] \label{pro:lip-one-phase}
Let $u: D \to \mathbb{R}$ be a minimizer of $J_{\mathrm{TP}}$ in $D$. 
There there is constant $C=C(n, p, \pm\lambda)$ such that if $x_0\in \Gamma^+_{\mathrm{OP}}$ (or $x_0 \in \Gamma^-_{\mathrm{OP}}$) is one-phase point and $B_r(x_0) \cap \Omega^-_u = \emptyset$ ($B_r(x_0) \cap \Omega^+_u = \emptyset$), then 
\[
\norm{\nabla u}_{L^{\infty}(B_{\frac{r}{2}}(x_0))} \le C.
\]
\end{proposition}

We remark that the condition $B_r(x_0) \cap \Omega^-_u = \emptyset$ always holds for some $r>0$ by the definition of one-phase points.
\begin{proof}
We know that $u_{x_0,r}(x)=\frac{u(x_0+rx)}{r}$ is a minimizer of the following one phase functional in $B_1$, i.e. minimizer of
$$ J_{\mathrm{OP}}(v,B_1):=\int_{B_1} |\nabla v|^p \, dx + (p-1) \lambda_+^p |\{v>0\} \cap B_1|, $$
over the class of nonnegative functions. 
Then the boundedness of the gradient 
$$\norm{\nabla u}_{L^{\infty}(B_{\frac{r}{2}}(x_0))}= \|\nabla u_{x_0,r} \|_{L^{\infty}(B_{\frac{1}{2}})} \leq C(n,p, \lambda_+), $$ 
follows from \cite[Theorem 3.3]{MR2133664}. 
We shall remark that the Lipschitz constant for one-phase problems does not depend on the boundary values of the minimizer as long as we stay uniformly far from the boundary.
\end{proof}

Next, we mention the following continuity result for minimizers. 

\begin{proposition}[BMO estimates for the gradient]
\label{P0}
Let $u$ be a minimizer of $J_{\mathrm{TP}}$ and $D' \Subset D$. Then,
\begin{enumerate}
\item[(i)]
for $1<p < 2$, we have that $|\nabla u|^{\frac{p-2}{2}} \nabla u \in BMO(D')$, and consequently $u \in C^{\sigma}(D')$ for any $\sigma \in (0,1)$;
\item[(ii)]
for $ 2 < p < \infty$, we have that $\nabla u \in BMO(D')$  and thus $u$ is locally log-Lipschitz continuous.
\end{enumerate}
In particular, $\nabla u \in L^{q}(D')$ for any $1<q<\infty$ and for any $1<p<\infty$.
\end{proposition}
\begin{proof}
The proof is the same as the proof of Lemma 3.1 in \cite{MR3771123}.
\end{proof}

The BMO estimate for the gradient of minimizers is sufficient to obtain the following compactness result.
Since we have not yet proved the Lipschitz continuity, this result will be extremely valuable for our argument in the next section.
We postpone the proof to Appendix \ref{appendix}.

\begin{proposition}
\label{LConverg}
Let $u_j$ be a bounded minimizer of $J_{\mathrm{TP}}$ in $B_2$ with the points $x_j \in B_1$ such that $u_j(x_j)=0$. 
Also, set $v_j(x)=\frac{u_j(x_j+r_jx)}{S_j}$, for any $x \in B_R$, with $0<R<\frac{1}{r_j}$, where $r_j \to 0$, as $j \to + \infty$ and $S_j >0$. Then, $v_j$ is the minimizer (according to its own boundary values) of the following scaled functional
\begin{equation}
\label{LEq17}
\hat{J}_{\mathrm{TP}}(v):=\int_{B_R} |\nabla v |^p + (p-1) \sigma_j^p  \lambda_+^p \chi_{\{v>0\}}+  (p-1)\sigma_j^p  \lambda_-^p \chi_{\{v<0\}}\, dx,
\end{equation}
where  $\sigma_j:=\frac{r_j}{S_j}$.
Moreover, if $|v_j| \leq M$ in $B_R$, for any fixed $0<R<\frac{1}{r_j}$ and for some $M=M(R)>0$, then up to a subsequence, the followings hold:  
\begin{enumerate}
\item [(i)]
For any $q>1$, and some $\alpha \in (0,1)$ (if $q>n$, one can take $\alpha=1-\frac{n}{q}$), $v_j$ converges to some function $v_0$ as $j \to + \infty$ in $C^{\alpha}(B_R)$ and weakly in $W^{1,q}(B_R)$;
\item [(ii)]
$ v_j \to  v_0$ strongly in $W^{1,p}(B_R)$; 
\item [(iii)]
If moreover, $\sigma_j:=\frac{r_j}{S_j} \to \sigma$, as $j \to + \infty$, then $v_0$ is a minimizer of 
\[
\hat{J}_{\mathrm{TP}}(v):=\int_{B_R} |\nabla v |^p +  (p-1) \sigma^p \lambda_+^p \chi_{\{v>0\}}+   (p-1)\sigma^p \lambda_-^p \chi_{\{v<0\}}\, dx.
\]
In particular, if $\sigma=0$, then $v_0$ is $p$-harmonic in $B_R$.
\end{enumerate}
\end{proposition}

\medskip

The following lemma states that $u^+$ and $u^-$ have coherent growth at two-phase points. 
This is essential to show that the minimizers are the viscosity solution of \eqref{OVERDETERMINED}.

\begin{lemma}
\label{LLe1}
Let $u$ be a bounded minimizer of $J_{\mathrm{TP}}$. Let $x_0 \in \Gamma_{\mathrm{TP}}$ and $r_0>0$ be small such that $B_{r_0}(x_0) \subset D$. Assume that $\sup_{B_r(x_0)} u^- \leq C_0r$ (resp. $\sup_{B_r(x_0)} u^+ \leq C_0 r$) for all $r \in (0,r_0]$. 
Then there exist constant $C_1>0$ such that $\sup_{B_r(x_0)} u^+ \leq C_1 r$ (resp. $sup_{B_r(x_0)} u^- \leq C_1 r$) for all $r \in (0,r_0]$.
\end{lemma}

\begin{proof}
We will just demonstrate one of the claims; the other can be demonstrated similarly. By the assumption of the lemma 
\begin{equation}
\label{LEq1}
\sup_{B_{r}(x_0)} u^- \leq C_0 r, \qquad \forall r \in (0,r_0].
\end{equation}
We claim that there is $\tilde C_1>0$ such that
\begin{equation}
\label{LEq2}
S(k+1)\leq \max \left\{ \frac{\tilde C_1}{2^{k+1}}, \frac{1}{2} S(k) \right\},
\end{equation}
where $S(k):=\norm{u}_{L^\infty(B_{2^{-k}}(x_0))}$, for any $k \in \mathbb{N}$ that $2^{-k}\le r_0$. 
To prove this, we argue by contradiction and suppose that \eqref{LEq2} fails. 
Then there is a sequence of integers $k_j$, with $j=1,2, \cdots$ such that
\begin{equation}
\label{LEq3}
S(k_j+1) > \max \left\{ \frac{j}{2^{k_j+1}}, \frac{1}{2} S(k_j) \right\}.
\end{equation}
Observe that since $u$ is a bounded minimizer, then \eqref{LEq3} implies that $k_j \to +\infty$ as $j \to +\infty$. Also, notice that \eqref{LEq3} implies that
\begin{equation}
\label{LEq4}
\sigma_j:=\frac{2^{-k_j}}{S(k_j+1)} \leq \frac{2}{j} \to 0, \qquad \text{as \,\,\, $j \to +\infty$.}
\end{equation}
Now, we introduce the scaled functions $v_j(x):=\frac{u(x_0+2^{-k_j}x)}{S(k_j+1)}$, for any $x \in B_1$. Then, from \eqref{LEq1} and \eqref{LEq4}, it follows that $v_j(0)=0$ and 
\begin{equation}
\label{LEq5}
v_j^-(x)=\frac{u^-(x_0+2^{-k_j}x)}{S(k_j+1)} \leq \frac{2^{-k_j}C_0}{S(k_j+1)} \leq \frac{2C_0}{j} \to 0, \qquad \text{as \,\,\, $j \to +\infty$.}
\end{equation}
Furthermore, it is simple to show that \eqref{LEq3} implies that
\begin{equation}
\label{LEq6}
\sup_{B_1} |v_j| \leq 2, \qquad \text{and} \qquad \sup_{B_{\frac{1}{2}}} |v_j| =1.
\end{equation}
Also,  Proposition \ref{LConverg} entails that $v_j$ is a minimizer of the scaled functional \eqref{LEq17} for $R=\frac{3}{4}$ and we can extract a converging subsequence such that $v_j \to v_0$ uniformly in $\overline{B_{\frac{3}{4}}}$ that $v_0$ is $p$-harmonic. 
The uniform convergence of $v_j$ to $v_0$ along with \eqref{LEq5}, \eqref{LEq6}, give that
$$ \Delta_p v_0(x)=0, \quad v_0(x) \geq 0 \,\,\, \text{if $x \in B_{\frac{3}{4}}$}, \qquad v(0)=0, \quad \sup_{B_{\frac{1}{2}}} v_0=1, $$
which is in contradiction with the strong minimum principle. Thus \eqref{LEq2} obtains.

Now we show how \eqref{LEq2} implies the lemma. Assume that $k_0$ is the smallest integer $k$ that $2^{-k} \le r_0$.
Let $\bar C_1 =\max( \tilde C_1, 2^{k_0} S(k_0))$. 
It is not difficult to  observe from \eqref{LEq2} that $S(k) \le \bar C_1 2^{-k}$.
For an arbitrary $r \in (0,r_0]$ choose $k \geq k_0$ such that $2^{-(k+1)}< r \le 2^{-k}$, then
\[
\norm{u}_{L^\infty(B_{r}(x_0))} \le \norm{u}_{L^\infty(B_{2^{-k}}(x_0))} = S(k) \le  \bar C_1 2^{-k} \le 2 \bar C_1 r.
\]
Thus the statement in the lemma holds for $C_1 = 2 \bar C_1$.
\end{proof}

The following theorem roughly says that, in a very weak sense, the free boundary conditions \eqref{OVERDETERMINED} hold.

\begin{proposition}
\label{LTh001}
Suppose that $u$ is a minimizer of $J_{\mathrm{TP}}$ in $D$ and $D' \subset D$ be such that $|D' \cap \{u=0\}|=0$. Then, we have the following free boundary condition in the very weak sense
$$ \lim_{\epsilon \to 0^+} \int_{\partial\{u>\epsilon\} \cap D'} \left( |\nabla u^+|^p-\lambda_+^p \right) \eta \cdot \nu + \lim_{\delta \to 0^+} \int_{\partial\{u<-\delta\} \cap D'} \left( |\nabla u^-|^p-\lambda_-^p \right) \eta \cdot \nu =0, $$
for any $\eta \in W_0^{1,p}(D', \mathbb{R}^n)$, and where $\nu$ is the outward normal.
\end{proposition}
\begin{proof}
The proof can be established precisely as in \cite[Theorem 2.4]{MR732100}.
\end{proof}

\begin{corollary}\label{cor:bndry-condition-global-sol}
Suppose $u(x)=\alpha \left(x \cdot \textbf{e} \right)^+-\beta \left( x \cdot \textbf{e} \right)^-$ is  a global minimizer of $J_{\mathrm{TP}}$
for some unite vector $\textbf{e}\in \mathbb{S}^{n-1}$ and the positive constants $\alpha$ and $\beta$.
Then $\alpha$ and $\beta$ satisfy conditions \eqref{E0-Con}.
\end{corollary}
\begin{proof}
The equality $ \alpha^p-\beta^p= \lambda_+^p-\lambda_-^p$ is obvious by invoking Proposition \ref{LTh001}.
Besides, conditions
$$ \alpha \geq \lambda_{+}, \qquad \text{and} \qquad \beta \geq \lambda_{-}, $$
can be obtained by a smooth variation of the free boundary $\{u=0\}=\{x\cdot \textbf{e} =0\} $. 
Indeed, by considering competitors of the form $u_t(x)=u^+(x)-u^-(x+t\xi(x))$ for vector fields $\xi \in C_c^\infty( \bR^n; \bR^n)$ with $\xi \cdot \textbf{e} \leq 0$ so that it moves negative phase only inwards, that is, $\{u_t<0\} \subset \{u<0\}$, and taking the derivative of $J_{\mathrm{TP}}(u_t, B_R)$ at $t>0$ and letting $t \to 0$ (where $R$ is sufficiently large such that $\supp \xi \subset B_R$), we get
$$ \int_{\{u=0\} \cap B_R} ( \xi \cdot \textbf{e}) \left( |\nabla u^-|^p- \lambda_-^p \right) \leq 0, $$
which gives $\beta \geq \lambda_-$. The estimate on $\alpha$ is analogous.
\end{proof}

\medskip

\section{Free boundary conditions in the viscosity sense}
Let $u : D \to \mathbb{R}$ be a local minimizer of $J_{\mathrm{TP}}$. 
In this section, we will show that $u$ satisfies the free boundary conditions \eqref{OVERDETERMINED}  in a weak (viscosity) sense.

\begin{definition}
\label{D1}
Let $D$ be an open set. 
We say that a function $Q:D \to \mathbb{R}$ touches a function $w: D \to \mathbb{R}$ from below (resp. from above) at a point $x_0 \in D$ if $Q(x_0)=w(x_0)$ and
$$ Q(x)-w(x) \leq 0 \qquad (resp. \,\,\, Q(x)-w(x) \geq 0), $$
for every $x$ in a neighborhood of $x_0$. We will say that $Q$ touches $w$ strictly from below (resp. above) if the above inequalities are strict for $x \neq x_0$.

A function $Q$ is an (admissible) comparison function in $D$ if
\begin{enumerate}
\item [(a)]
$Q \in C^1(\overline{\{Q>0\}} \cap D) \cap C^1(\overline{\{Q<0\}} \cap D)$;
\item [(b)]
$Q \in C^2(\{Q>0\} \cap D) \cap C^2(\{Q<0\} \cap D)$;
\item [(c)]
$\partial \{Q>0\}$ and $\partial \{Q<0\}$ are smooth manifolds in $D$.
\end{enumerate}
We should remark that if $\nabla Q \neq 0$ on $\partial \{Q>0\} \cup \partial\{Q<0\}$, the condition $(c)$ above holds.

\end{definition}

\begin{lemma}
\label{L2}
Let $u$ be a local minimizer of $J_{\mathrm{TP}}$ in the open set $D \subset \mathbb{R}^n$. Then the following optimality conditions on the free boundary $F(u)$ hold.
\begin{enumerate}
\item [\textup{(A)}] Suppose that $Q$ is a comparison function that touches $u$ from \underline{below} at $x_0$.
\begin{enumerate}
\item[\textup{(A.1)}] If $x_0 \in \Gamma^{+}_{\mathrm{OP}}$, then $|\nabla Q^+(x_0)| \leq \lambda_{+}$;
\item[\textup{(A.2)}]  if $x_0 \in \Gamma^{-}_{\mathrm{OP}}$, then $ Q^+ \equiv 0 $ in a neighborhood of $x_0$ and $|\nabla Q^-(x_0)| \geq \lambda_{-}$;
\item[\textup{(A.3)}]  if $x_0 \in \Gamma_{\mathrm{TP}}$, then $|\nabla Q^-(x_0)| \geq \lambda_{-}$ and 
$$ |\nabla Q^+ (x_0)|^p-|\nabla Q^-(x_0)|^p \leq \lambda_+^p-\lambda_-^p. $$
\end{enumerate}
\item [\textup{(B)}] Suppose that $Q$ is a comparison function that touches $u$ from \underline{above} at $x_0$.
\begin{enumerate}
\item [\textup{(B.1)}]  If $x_0 \in \Gamma^{+}_{\mathrm{OP}}$, then $ Q^- \equiv 0 $ in a neighborhood of $x_0$ and $|\nabla Q^+(x_0)| \geq \lambda_{+}$;
\item [\textup{(B.2)}]  if $x_0 \in \Gamma^{-}_{\mathrm{OP}}$, then $|\nabla Q^-(x_0)| \leq \lambda_{-}$;
\item [\textup{(B.3)}]  if $x_0 \in \Gamma_{\mathrm{TP}}$, then $|\nabla Q^+(x_0)| \geq \lambda_{+}$ and 
$$ |\nabla Q^+ (x_0)|^p-|\nabla Q^-(x_0)|^p \geq \lambda_+^p-\lambda_-^p. $$
\end{enumerate}
\end{enumerate}
\end{lemma}

\begin{proof} 
First, we will prove the gradient bounds in $\textup{(A.1)}$ and $\textup{(B.1)}$.
The case $x_0 \in \Gamma^-_{\mathrm{OP}}$, and the proofs of $\textup{(A.2)}$ and $\textup{(B.2)}$ can be obtained similarly.

Let $x_0 \in \Gamma^+_{\mathrm{OP}} $ be a one-phase point and let $Q$ touches $u$ from below at $x_0$. Then, $Q^+$ touches $u$ from below at $x_0$, too. Consider $u_{x_0,r_k}(x)=\frac{u(x_0+r_kx)}{r_k}$ and $Q^+_{x_0,r_k}(x)=\frac{Q^+(x_0+r_kx)}{r_k}$ as the blow-up sequences of $u$ and $Q^+$ at $x_0$. 
By virtue of Proposition \ref{pro:lip-one-phase}, the functions  $u_{x_0,r_k}$ are uniformly Lipschitz for  sufficiently $r_k$ small and 
up to extracting a subsequence, we can assume that $u_{x_0,r_k}$ converges uniformly to a blow-up limit $v$. 
The limit $v$ is a minimizer of one-phase functional  $J_{\mathrm{OP}}$ and so $\Delta_p v =0$ in $\{v>0\}$.

On the other hand, since $Q^+$ is differentiable at $x_0$ in $\overline{\Omega_Q^+}$, we get that $Q^+_{x_0,r_k}$ converges to the function
\begin{equation}
\label{EEEEEEE01}
H_{Q^+}(x)=\left( x \cdot \textbf{e}' \right)^+ \qquad \text{with \qquad $\textbf{e}'= \nabla Q^+(x_0)$}.
\end{equation}
If $\nabla Q^+(x_0)=0$, $\textup{(A.1)}$  is trivially valid. 
We assume that $\textbf{e}'\ne 0$, and since $H_{Q^+}$ touches $v$ from below at $x=0$, we get
\begin{equation*}
v(x)= \alpha (x\cdot \textbf{e}')^+ + o(|x|), \qquad \alpha \geq 1,
\end{equation*}
for some $\alpha$; see \cite[Lemma B.1]{fotouhi-shahgholian2023}. 
We get that any blow-ups of $v$ will be $v_0(x)= \alpha (x\cdot \textbf{e}')^+$ which is also a minimizer of  $J_{\mathrm{OP}}$. 
Thus $\alpha|\textbf{e}'|= \lambda_+$ due to the free boundary condition for one-phase minimizers (Proposition \ref{LTh001} or see \cite[Theorem 2.1]{MR2133664}) and so
$$ |\nabla Q^+(x_0)| \leq \lambda_+.$$

Similarly, when $Q$ touches $u$ from above at $x_0$, then also $Q^+$ touches $u$ from above at $x_0$, and the claim $Q^- \equiv 0$ in $\textup{(B.1)}$ is trivially true. 
Again, consider $u_{x_0,r_k}$, which up to extracting a subsequence, converges uniformly to a blow-up limit $v$ and $Q^+_{x_0,r_k}$, as the blow-up sequences of $Q^+$ at $x_0$, which converges to the function \eqref{EEEEEEE01}. 
Now, we argue similar to the proof of $\textup{(A.1)}$ to get
$$ |\nabla Q^+(x_0)| \geq \lambda_+. $$

Now, we prove $\textup{(A.3)}$. Suppose $ x_0 \in \Gamma_{\mathrm{TP}} $ and assume that $Q$ touches $u$ from below at $x_0$.
Then $u^- \le Q^-$ and $u^-(x) \leq C_0|x-x_0|$  for $C_0 = 2 |\nabla Q^-(x_0)|$ if $|x-x_0|$ is sufficiently  small. 
Now we  employ Lemma \ref{LLe1} to deduce that  $|u(x)| \le C_1|x-x_0|$ in a neighborhood of $x_0$.

Let $u_{x_0,r_k}$ and $Q_{x_0,r_k}$ be the blow-up sequences of $u$ and $Q$ at $x_0$. 
Then, by using Proposition  \ref{LConverg}, up to extracting a subsequence, we can assume that $u_{x_0,r_k}$ converges uniformly to some function $v$ which is also a minimizer of  $J_{\mathrm{TP}}$. Moreover, it satisfies
\begin{equation}
\label{EEEEE01}
|v(x)| \le C_1|x|.
\end{equation}

On the other hand, since $Q^+$ and $Q^-$ are differentiable at $x_0$ (respectively in $\overline{\Omega_Q^+}$ and $\overline{\Omega_Q^-}$), we get that $Q_{x_0,r_k}$ converges to the function
$$ H_Q(x)= \left( x \cdot \tilde{\textbf{e}}^+ \right)^+- \left( x \cdot \tilde{\textbf{e}}^- \right)^-, $$
where $\tilde{\textbf{e}}^\pm=\nabla Q^\pm(x_0)$.
Since $H_Q$ touches $v$ from below at $x=0$, we have (\cite[Lemma B.1]{fotouhi-shahgholian2023})
\begin{equation*}
\begin{split}
v^+(x)= \alpha (x\cdot \tilde{\textbf{e}}^-)^+ + o(|x|), & \qquad   |\tilde{\textbf{e}}^+ |  \leq \alpha |\tilde{\textbf{e}}^-|, \\
v^-(x) = \beta (x\cdot \tilde{\textbf{e}}^-)^- + o(|x|), & \qquad  \beta \le 1 ,
\end{split}
\end{equation*}
for some $\alpha, \beta \ge0$.  Note that by virtue of the non-degeneracy, Proposition \ref{P1.5}, $v^- \not \equiv 0$ and so  $\tilde{\textbf{e}}^- \ne 0$.
If $v_0$ is a blowup of $v$ (recall \eqref{EEEEE01} and Proposition \ref{LConverg}), it will be  
$v_0(x) = \alpha (x\cdot \tilde{\textbf{e}}^-)^+ - \beta (x\cdot \tilde{\textbf{e}}^-)^-$ which is also a minimizer of $J_{\mathrm{TP}}$. 
Now apply Corollary \ref{cor:bndry-condition-global-sol}, we get 
\[
(\alpha^p -\beta^p) |\tilde{\textbf{e}}^-|^p = \lambda_+^p -\lambda_-^p, \qquad \beta|\tilde{\textbf{e}}^-| \ge \lambda_-.
\]
Hence,
$$ |\nabla Q^+ (x_0)|^p-|\nabla Q^-(x_0)|^p \leq \alpha^p-\beta^p=\lambda_+^p-\lambda_-^p,$$ 
as well as $|\nabla Q^-(x_0)| \ge \lambda_-$.
The proof of $\textup{(B.3)}$ is analogous.
\end{proof}

If $u : D \to \mathbb{R}$ is a continuous function such that the claims $\textup{(A)}$ and $\textup{(B)}$ hold for every comparison function $Q$, then we say that $u$ satisfies the boundary condition \eqref{OVERDETERMINED} on the free boundary in viscosity sense.

We need the following straightforward consequence of the definition of viscosity solution. It emphasizes what happens when a function is touching only one of the two phases.

\begin{lemma}
\label{L3}
Let $u: D \to \mathbb{R}$ be a continuous function that satisfies \eqref{OVERDETERMINED}. 
\begin{enumerate}
\item [\textup{(i)}] Assume that $Q$ is a comparison function touching $u^+$ from above at $x_0 \in \partial \Omega^+_u$ (resp. $-u^-$ from below at $x_0 \in \partial \Omega^-_u$), then
$$ |\nabla Q^+(x_0)| \geq \lambda_{+} \qquad \left(resp. \,\,\, |\nabla Q^-(x_0)| \geq \lambda_{-} \right). $$
\item [\textup{(ii)}] Assume that $Q$ is a comparison function touching $u^+$ from below at $x_0 \in \Gamma^{+}_{\mathrm{OP}}$ (resp. $-u^-$ from above at $x_0 \in \Gamma^{-}_{\mathrm{OP}}$), then
$$ |\nabla Q^+(x_0)| \leq \lambda_{+} \qquad \left(resp. \,\,\, |\nabla Q^-(x_0)| \leq \lambda_{-} \right). $$
\end{enumerate}
\end{lemma}
\begin{proof}
Statement $\textup{(i)}$ will be obtained directly from $\textup{(B)}$. 
The proof of $\textup{(ii)}$ follows the same lines of arguments as the proof of Lemma \ref{L2}.
\end{proof}

\section{Flatness decay at two-phase points}
In this section, we will  follow the method of  improvement of flatness. In fact, we will prove that at two-phase points $x_0 \in \Gamma_{\mathrm{TP}}$, there is a constant $\epsilon_0>0$ such that if $u$ is $\epsilon_0$-flat in $B_r(x_0)$ with respect to $H=H_{\alpha,\textbf{e}}$, then it has excess flatness in smaller scales with respect to another $\tilde{H}=H_{\tilde{\alpha},\tilde{\textbf{e}}}$.

\begin{theorem}
\label{T2}
For every $1<p<\infty$, $ 0<L_{0}, L_{1} $ and $\gamma \in (0,\frac{1}{2})$, there exist $\epsilon_0>0$, $C>0$ and $\rho >0 $ such that if the function $u: B_1 \to \mathbb{R}$ satisfies:
\begin{enumerate}[label=(\alph*)]
\item the origin is on the two-phase free boundary, $0 \in \Gamma_{\mathrm{TP}}$;
\item $u$ is $p$-harmonic in $\Omega^+_u \cup \Omega^-_u$;
\item $u$ satisfies the free boundary condition \eqref{OVERDETERMINED} in viscosity sense;
\item $u$ is $\epsilon_0$-flat in $B_1$, that is,
\begin{equation}
\label{E4.5}
\|u-H_{\alpha, \textbf{e}_n}\|_{L^{\infty}(B_1)} \leq \epsilon_0, \qquad \text{for some \,\,\, $ \max(\lambda_{+},L_0) \le  \alpha \leq L_1$},
\end{equation}
\end{enumerate}
then, there are $\textbf{e} \in \mathbb{S}^{n-1}$ and $\tilde{\alpha} \geq \max(\lambda_{+},L_0)$, such that
\begin{equation}
\label{E5}
|\textbf{e}-\textbf{e}_n|+|\tilde{\alpha}-\alpha| \leq C \|u-H_{\alpha, \textbf{e}_n}\|_{L^{\infty}(B_1)},
\end{equation}
and 
\begin{equation}
\label{E6}
\|u_{\rho}-H_{\tilde{\alpha}, \textbf{e}}\|_{L^{\infty}(B_1)} \leq \rho^{\gamma} \|u-H_{\alpha, \textbf{e}_n}\|_{L^{\infty}(B_1)},
\end{equation}
where $u_{\rho}(x)$ denotes $u_{0,\rho}(x)=\frac{u(\rho x)}{\rho}$.
\end{theorem}

Theorem \ref{T2} is an easy consequence of the two upcoming lemmas. In the first one, we deal with the situation where the two-plane is, roughly, $H_{\lambda_{+}, \textbf{e}}$ for some $\textbf{e} \in \mathbb{S}^{n-1}$. Note that this is the case where one might expect the presence of branching points and it is indeed in this setting that we will obtain the two membrane problems as "linearization" (see e.g. \cite[Subsection 1.3]{MR4285137} for a presentation of the linearization method in studying the regularity of free boundaries). In the second lemma, we deal with the case when the closest half-plane solution has a gradient much larger than $\lambda_{+}$. In this case, the origin will be an interior two-phase point. In fact, in one-phase problems, it is possible to obtain universal interior bounds, in the sense that, if $u$ is a solution in a ball $B_1$ and $0 \in F(u)$, then $| \nabla u|$ is bounded in $B_{\frac{1}{2}}$ by a universal constant, no matter what the boundary data are. However, in two-phase problems, this is generally not possible. For instance, in the one-dimensional minimization scenario, increasing the boundary data leads to the appearance of a solution with a large gradient near the origin, see \cite[Section 1.1]{MR2145284}.

\begin{lemma}[Improvement of flatness: branching points]
\label{L4}
For every $1<p<\infty$, $ 0<L_0, L_1 $, $\gamma \in (0,\frac{1}{2})$, and $M>0$, there exist $\epsilon_1=\epsilon_1(p,\gamma, n, L_0, L_1, M)$, $C_1=C_1(p,\gamma, n, L_0, L_1, M)$ and $\rho=\rho(p,\gamma, n, L_0, L_1, M)$ such that 
if function $u: B_1 \to \mathbb{R}$ satisfies $(a)-(b)-(c)$ of Theorem \ref{T2} and furthermore
$$ \|u-H_{\alpha, \textbf{e}_n}\|_{L^{\infty}(B_1)} \leq \epsilon_1, \qquad \text{with} \qquad L_0 \le \lambda_{+} \leq \alpha \leq \lambda_{+}+M\|u-H_{\alpha, \textbf{e}_n}\|_{L^{\infty}(B_1)}, $$
then there exist $\textbf{e} \in \mathbb{S}^{n-1}$ and $\tilde{\alpha} \geq \lambda_{+}$, for which \eqref{E5} and \eqref{E6} hold.
\end{lemma}

\begin{lemma}[Improvement of flatness: non-branching points]
\label{L5}
For every $1<p<\infty$, $ 0<L_0,  L_1 $ and $\gamma \in (0,1)$, there exist $\epsilon_2=\epsilon_2(p,\gamma, n, L_0, L_1)$, $\overline{M}=\overline{M}(p,\gamma, n, L_0, L_1)$, $\rho=\rho(p,\gamma, n, L_0, L_1)$ and $C_2=C_2(p,\gamma, n, L_0, L_1)$ such that 
if function $u: B_1 \to \mathbb{R}$ satisfies $(a)-(b)-(c)$ of Theorem \ref{T2} and furthermore
$$ \|u-H_{\alpha, \textbf{e}_n}\|_{L^{\infty}(B_1)} \leq \epsilon_2, \qquad \text{with} \qquad \alpha \geq \max( \lambda_{+}, L_0) + \overline{M} \|u-H_{\alpha, \textbf{e}_n}\|_{L^{\infty}(B_1)}, $$
then there exist $\textbf{e} \in \mathbb{S}^{n-1}$ and $\tilde{\alpha} \geq \max( \lambda_{+}, L_0)$, for which \eqref{E5} and \eqref{E6} hold.
\end{lemma}

\begin{proof}[Proof of Theorem \normalfont{\ref{T2}}]
The proof follows easily by combining the Lemmas \ref{L4} and \ref{L5}.
\end{proof}

In order to prove Lemma \ref{L4} and Lemma \ref{L5}, we will argue by contradiction. Hence in the following, we consider a sequence $u_k$ of minimizers such that
\begin{equation}
\label{E7}
\epsilon_k:= \|u_k-H_{\alpha_k, \textbf{e}_n}\|_{L^{\infty}(B_1)}  \to 0, \qquad \text{and} \qquad \lambda_{+} \leq \alpha_k \leq L_1.
\end{equation}
We also set 
\begin{equation}
\label{E8}
\ell:= \lambda_{+}^{p} \lim_{k \to \infty} \frac{\alpha_k^{p}- \lambda_{+}^p }{p\alpha_k^{p} \epsilon_k}= \lambda_{-}^{p} \lim_{k \to \infty} \frac{\beta_k^p- \lambda_{-}^p }{p\beta_k^{p} \epsilon_k},
\end{equation}
which we can assume to exist up to  a subsequence. It might be useful to keep in mind that $\ell=\infty$ will correspond to Lemma \ref{L5} while $0 \leq \ell < \infty$ (so $\alpha_k\to \lambda_+$ and $\lambda_+ \ge L_0$)  to Lemma \ref{L4}. 

We first show that the sequence
\begin{equation}
\label{E9}
v_k(x)=
\begin{cases}
v_{+,k}(x):=\dfrac{u_k(x)-\alpha_k x_n^+}{ \epsilon_k\alpha_k}, \qquad & x \in \Omega_{u_k}^+ \cap B_1 \\
\\
v_{-,k}(x):=\dfrac{u_k(x)+\beta_k x_n^-}{ \epsilon_k\beta_k}, \qquad & x \in \Omega_{u_k}^- \cap B_1
\end{cases}
\end{equation}
is compact in some suitable sense. 
This will be mentioned in Lemma \ref{L5.5} below and the proof will come in Subsection \ref{3-1}. 
Then, in Lemma \ref{L6}, we obtain the limiting problem which is solved by $v$, the limit of $v_k$. 
Finally, in Subsection \ref{3-3} we show how to deduce Lemma \ref{L5} and Lemma \ref{L4} from Lemma \ref{L5.5} and Lemma \ref{L6}. 

In the following, we will denote with
$$ B^{\pm}_r:=B_r \cap \{x_n^{\pm}>0\}, \qquad \text{for every $r>0$}. $$

\begin{lemma}[Compactness of the linearizing sequence $v_k$]
\label{L5.5}
Let $u_k$ be a sequence of functions satisfying $(a)-(b)-(c)$ of Theorem \ref{T2} uniformly in $k$ and let $\epsilon_k$ and $\alpha_k$ be as in \eqref{E7} and let $v_k$ be defined by \eqref{E9}. Then there are H\"older continuous functions 
$$ v_+: \overline{B^+_{\frac{1}{2}}} \to \mathbb{R}, \qquad \text{and} \qquad v_-: \overline{B^-_{\frac{1}{2}}} \to \mathbb{R}, $$
with
$$ v_+ \leq v_- \qquad \text{on} \quad B_{\frac{1}{2}} \cap \{ x_n=0\}, \qquad v_+(0)=v_-(0)=0, $$
and such that the sequence of closed graphs
$$ \Gamma^{\pm}_k:=\left\{ (x,v_{\pm,k}(x)) \, : \, x \in \overline{\Omega^{\pm}_{u_k} \cap B_{\frac{1}{2}}} \right\}, $$
converge, up to a  subsequence, in the Hausdorff distance to the closed graphs
$$ \Gamma_{\pm}=\left\{ (x,v_{\pm}(x)) \, : \, x \in \overline{B^{\pm}_{\frac{1}{2}}} \right\}. $$
In particular, the following claims hold.
\begin{enumerate}[label=(\roman*)]
\item For every $\delta>0$, $v_{\pm,k}$ converges uniformly to $v_{\pm}$ on $B_{\frac{1}{2}} \cap \{\pm x_n>\delta\}$.
\item For every sequence $x_k \in \overline{\Omega^{\pm}_{u_k}} \cap B_1 $ converging to $x \in \overline{B^{\pm}_{\frac{1}{2}}}$, we have 
$$ v_{\pm}(x)=\lim_{k \to \infty} v_{\pm,k}(x_k). $$
\item For every $x \in \{x_n=0\} \cap B_{\frac{1}{2}}$, we have 
$$ v_{\pm}(x)=- \lim_{k \to \infty} \frac{x_k \cdot \textbf{e}_n}{\epsilon_k}, \qquad \text{for any sequence \,\,\, $\partial \Omega^{\pm}_{u_k} \ni x_k \to x$}. $$
\end{enumerate}
In particular, $\{x_n=0\} \cap \overline{B_{\frac{1}{2}}}$ decomposes into an open jump set
$$ \mathcal{J}=\{v_+<v_-\} \cap \{x_n=0\} \cap \overline{B_{\frac{1}{2}}}, $$
and its complementary contact set
$$ \mathcal{C}=\{v_+=v_-\} \cap \{x_n=0\} \cap \overline{B_{\frac{1}{2}}}. $$
Furthermore, if $x \in \mathcal{J}$, then 
\begin{equation}
\label{E10}
\liminf_{k \to \infty} \mathrm{\,dist} \left(x, \partial \Omega^+_{u_k} \cap \partial \Omega^-_{u_k} \right) >0.
\end{equation}
In particular for all $x \in \mathcal{J}$, there exists two sequences $x^{\pm}_k \in \Gamma^{\pm}_{k,\mathrm{OP}}$ such that $x^{\pm}_k \to x$.
\end{lemma}
Now, in the next lemma, we determine the limiting problem for the function $v$ which is defined as
\begin{equation}
\label{E11}
v(x)=
\begin{cases}
v_{+}(x), &  \qquad \text{for $x \in B^+_{\frac{1}{2}}$},  \\
v_{-}(x), &  \qquad \text{for $x \in B^-_{\frac{1}{2}}$},
\end{cases}
\end{equation}
where $v_+$ and $v_-$ are the functions defined in Lemma \ref{L5.5}.

In what follows, we will denote with
$$ \mathcal{L}_p(u):= \Delta u + (p-2) \partial_{nn} u, $$
the frequently used operator which appears in the linearized problem. 

\begin{lemma}[The "linearized" problem]
\label{L6}
Let $u_k$, $\epsilon_k$ and $\alpha_k$ be as in \eqref{E7}, $v_k$ be defined by \eqref{E9} and $\ell$ as in \eqref{E8}. Let also $v_{\pm}$ be as in Lemma \ref{L5.5}:

If \textbf{ $\ell=\infty$}, then $\mathcal{J}=\emptyset$ and $v_{\pm}$ are viscosity solutions of the following transmission problem:
\begin{equation}
\label{E12}
\begin{cases}
\mathcal{L}_p(v_{\pm})=\Delta v_{\pm} + (p-2) \partial_{nn} v_{\pm} = 0, \qquad & \mathrm{in} \quad B^{\pm}_{\frac{1}{2}}, \\
\alpha_{\infty}^{p} \partial_n v_+ = \beta_{\infty}^{p} \partial_n v_-, \qquad & \mathrm{on} \quad B^{\pm}_{\frac{1}{2}} \cap \{x_n=0\},
\end{cases}
\end{equation}
where $\alpha_{\infty}=\lim_{k \to \infty} \alpha_k$ and $\beta_{\infty}=\lim_{k \to \infty} \beta_k$, which we can assume to exist up to extracting a further subsequence. 

If \textbf{$0 \leq \ell <\infty$}, then $v_{\pm}$ are viscosity solutions of the following two membranes problem:
\begin{equation}
\label{E13}
\begin{cases}
\mathcal{L}_p(v_{\pm})=\Delta v_{\pm} + (p-2) \partial_{nn} v_{\pm} = 0, \qquad & \mathrm{in} \quad B^{\pm}_{\frac{1}{2}}, \\
\lambda_{\pm}^{p} \partial_n v_{\pm} + \ell \geq 0,  \qquad & \mathrm{in} \quad B_{\frac{1}{2}} \cap \{x_n=0\}, \\
\lambda_{\pm}^{p} \partial_n v_{\pm} + \ell = 0,  \qquad & \mathrm{in} \quad \mathcal{J}, \\
\lambda_{+}^{p} \partial_n v_+ = \lambda_{-}^{p} \partial_n v_-, \qquad & \mathrm{in} \quad \mathcal{C}, \\
v_+ \leq v_-, \qquad & \mathrm{in} \quad B_{\frac{1}{2}} \cap \{x_n=0\}. \\
\end{cases}
\end{equation}
\end{lemma}

\begin{remark}
\label{R1}
Here by viscosity solution of \eqref{E12} and \eqref{E13}, we mean a function $v$ as in \eqref{E11} such that $v_{\pm}$ are continuous in $\overline{B^{\pm}_{\frac{1}{2}}}$, $\mathcal{L}_p(v_{\pm})=0$ in $B^{\pm}_{\frac{1}{2}}$ (in viscosity or equivalently the classical sense) and such that the following holds.
\begin{itemize}
\item If we are in case \eqref{E12}, let $s,t \in \mathbb{R}$ and let $\tilde{P}$ be a quadratic polynomial such that $\partial_n \tilde{P}=0$. Suppose that $\mathcal{L}_p(\tilde{P}) \geq 0$ ($\mathcal{L}_p(\tilde{P}) \leq 0$) and that the function 
$$ P:=sx_n^+-tx_n^-+\tilde{P}, $$
touches $v$ strictly from below (above) at a point $x_0 \in B_{\frac{1}{2}} \cap \{x_n=0\}$, then
$$ \alpha^{p}_{\infty} s \leq \beta^{p}_{\infty} t, \qquad \left( \alpha^{p}_{\infty} s \geq \beta^{p}_{\infty} t \right). $$
\item If we are in case \eqref{E13} then
\begin{enumerate}
\item[(1)] if $P_{\pm}$ is a quadratic polynomial with $\mathcal{L}_p(P_{\pm}) \leq 0$ in $B^{\pm}_{\frac{1}{2}}$ touching $v_{\pm}$ strictly from above at $x_0 \in B_{\frac{1}{2}} \cap \{x_n=0\}$, then $\lambda^{p}_{\pm} \partial_n P_{\pm} \geq 0$;
\item[(2)] if $P_{\pm}$ is a quadratic polynomial with $\mathcal{L}_p(P_{\pm}) \geq 0$ in $B^{\pm}_{\frac{1}{2}}$ touching $v_{\pm}$ strictly from below at $x_0 \in \mathcal{J}$, then $\lambda^{p}_{\pm} \partial_n P_{\pm} \leq 0$;
\item[(3)] if $s,t \in \mathbb{R}$ and $\tilde{P}$ is a quadratic polynomial with $\mathcal{L}_p(P_{\pm}) \geq 0$ ($\mathcal{L}_p(P_{\pm}) \leq 0$) such that $\partial_n \tilde{P} = 0$ and  the function
$$P:=sx_n^+-tx_n^-+\tilde{P},$$
touches $v$ strictly from below (above) at a point $x_0 \in B_{\frac{1}{2}} \cap \{x_n=0\}$, then 
$$ \lambda^{p}_{+} s \leq \lambda^{p}_{-} t, \qquad \left( \lambda^{p}_{+} s \geq \lambda^{p}_{-} t \right). $$
\end{enumerate}
\end{itemize}
\end{remark}

\subsection{Compactness of the linearizing sequence}
\label{3-1}
As explained in \cite[Subsection 3.1]{MR4285137} for the case of classical two-phase Bernoulli problem, the authors declare that the key point in establishing suitable compactness for $v_k$ is a "partial Harnack" inequality. 
We will follow the same approach and  start with the following useful lemma.

\begin{lemma}
\label{L7}
There is a constant $\tau=\tau(n,p)>0$ such that the following holds. Assume that $v: B_1 \to \mathbb{R}$ is a continuous function with $\Delta_p v = 0$ in $\{v>0\}$ and   
$$ \lambda \left(x_n+b \right)^+\leq v(x) \leq \lambda \left(x_n+a \right)^+, \qquad x\in B_1,$$
for some $\lambda>0$ and $a,b \in (-\frac{1}{100}, \frac{1}{100})$. Let $P=(0, \cdots, 0, \frac{1}{2})$, then for all $\epsilon \in (0,\frac{1}{2})$
$$ v(P) \leq \lambda (1-\epsilon) \left(\frac{1}{2}+a \right)^+ \quad \Longrightarrow \quad v(x) \leq \lambda (1-\tau \epsilon) \left(x_n+a \right)^+ \qquad \mathrm{in} \quad B_{\frac{1}{4}}(0), $$
and
$$ v(P) \geq \lambda (1+\epsilon) \left(\frac{1}{2}+b \right)^+ \quad \Longrightarrow \quad v(x) \geq \lambda (1+\tau \epsilon) \left(x_n+b \right)^+ \qquad \mathrm{in} \quad B_{\frac{1}{4}}(0). $$
\end{lemma}

\begin{proof}
We prove only the first implication since the second statement can be obtained by the same arguments. First, we notice that, since $|b| < \frac{1}{100}$, both $v$ and $\lambda(x_n+a)^+$ are positive and $p$-harmonic in $B_{\frac{1}{4}}(P)$. Thus,
$$ \lambda (x_n+a)^+ - v(x) \geq 0, \qquad x\in B_{\frac{1}{4}}(P), $$
and
$$ \lambda \left(\frac{1}{2}+a \right)^+-v(P) \geq \lambda \epsilon \left( \frac{1}{2}+ a \right)^+ \geq \frac{49}{100} \lambda \epsilon. $$

Now, we distinguish two cases:

\textbf{Case (i)}. Suppose $|\nabla v(P)| < \frac{\lambda}{4}$. 
Therefore, there exists $r_1=r_1(n,p)>0$ such that $|\nabla v(x)| \leq \frac{\lambda}{2}$ in $B_{4r_1}(P)$ (note that $\frac{v}{\lambda}$ is universally bounded and $p$-harmonic in $B_{\frac{1}{4}}(P)$).

It is easy to find that for $\tilde{v}:= (x_n+a)^+-\frac{1}{\lambda} v$, we have
$$ \mathrm{div} \left(|\nabla \tilde{v} - \textbf{e}_n|^{p-2} (\nabla \tilde{v} - \textbf{e}_n) \right)=0, \qquad \text{in} \quad B_{\frac{1}{20}}(P). $$
We now apply Harnack's inequality for the above operator (see e.g. \cite[Lemma 4.1]{MR4273843}) in $B_{4r_1}(P)$, to deduce that
$$ (x_n+a)^+-\frac{1}{\lambda} v(x) \geq C^{-1} \left( \left(\frac{1}{2}+a \right)^+-\frac{1}{\lambda} v(P) \right)-r_1, \qquad \text{in} \quad B_{r_1}(P), $$  
for an appropriate universal constant $C=C(n,p)>0$. On the other hand, for all $x \in B_{r_1}(P)$, we obtain
$$ 
\begin{aligned}
C^{-1} \frac{49}{100} \epsilon - r_1 & \leq (x_n+a)^+-\frac{1}{\lambda} v(x) \\
& \leq (x_n+2r_1+a)^+-2r_1-\frac{1}{\lambda} v(x+2r_1\textbf{e}_n) + \frac{2r_1}\lambda \norm{\nabla v}_{L^\infty(B_{4r_1}(P))} \\
& \leq (x_n+2r_1+a)^+-2r_1-\frac{1}{\lambda} v(x+2r_1\textbf{e}_n) + r_1 \\
& \leq (x_n+2r_1+a)^+-\frac{1}{\lambda} v(x+2r_1\textbf{e}_n) - r_1.
\end{aligned}
$$
Thus, with $\tilde{P}=P+2r_1\textbf{e}_n$, we get
\begin{equation}
\label{MUL1}
C^{-1} \frac{49}{100} \epsilon \leq  \left(x_n+a \right)^+-\frac{1}{\lambda} v\left(x \right), \qquad \text{for all \,\,\ $x \in B_{r_1}(\tilde{P}$).}
\end{equation}
Hence, by considering the inequality \eqref{MUL1} and also using the bound $|a| \leq \frac{1}{100}$, there is a constant $c=c(n,p)$ such that
$$ v(x) \leq \lambda (1-c \epsilon) (x_n+a)^+, \qquad \text{for all \,\,\, $x \in B_{r_1}(\tilde{P})$.} $$ 
We now let $w$ be the solution to the following problem
$$
\begin{cases}
\Delta_p w = 0, \qquad & \mathrm{in} \quad \left( B_1(0) \setminus B_{r_1}(\tilde{P}) \right) \cap \{x_n>-a\} \\
w=0, \qquad & \mathrm{on} \quad B_1 \cap \{x_n=-a\} \\
w=(x_n+a)^+,  \qquad & \mathrm{on} \quad \partial B_1(0) \cap \{x_n>-a\} \\
w=(1-c\epsilon)(x_n+a)^+, \qquad & \mathrm{on} \quad \partial B_{r_1}(\tilde{P}) \cap \{x_n>-a\}.
\end{cases}
$$
By the Hopf boundary lemma (\cite[Proposition 3.2.1]{MR700735}),
$$ w(x) \leq (1-\tau \epsilon)(x_n+a)^+, \qquad \text{for every $x \in B_{\frac{1}{4}} \cap \{x_n>-a\}$,} $$
for a suitable constant $\tau=\tau(n,p)$. 
On the other hand, by the comparison principle, we have $v \leq \lambda w$ in $\{v>0\}\cap B_1 \setminus B_{r_1}(\tilde{P})$, which concludes the proof in \textbf{Case (i)}. 

\medskip

\textbf{Case (ii)}. Suppose $|\nabla v(P)| \geq \frac{\lambda}{4}$. By the interior gradient estimate, we know that $\nabla v$ is bounded in $B_{\frac{1}{40}}(P)$, and  there exist a constant $0<r_0=r_0(n,p)$, with $8r_0 \leq \frac{1}{40}$ such that
$$ \frac{\lambda}{8} \leq |\nabla v(x)| \leq C\lambda, \qquad \text{for all \,\,\, $x \in B_{8r_0}(P)$,} $$
for an appropriate universal constant $C=C(n,p)>0$. Now, $v$ will be the weak solution to the following uniformly elliptic equation
$$ \sum_{i,j=1}^n  \theta_{ij}  \partial_{x_ix_j} v = 0, \qquad \text{in} \quad B_{4r_0}(P), $$
with $\theta_{ij}=\delta_{ij} + (p-2)|\nabla v|^{-2} \partial_{x_i}v \partial_{x_j}v $.
In this way, $\tilde{v}$ also will satisfy a uniformly elliptic equation in $B_{4r_0}(P)$, and then by applying Harnack's inequality (see e.g. \cite[Chapter 9]{MR1814364}), we get
\begin{equation}
\label{MUL2}
C^{-1} \frac{49}{100} \epsilon \leq \left(x_n+a \right)^+-\frac{1}{\lambda} v\left(x \right), \qquad \text{for all \,\,\, $x \in B_{r_0}(P)$}.
\end{equation}
Now, we can repeat the same argument of \textbf{Case (i)}, by considering the inequality \eqref{MUL2} in the ball $B_{r_0}(P)$ instead of inequality \eqref{MUL1}. This completes the proof of the lemma.
\end{proof}

We next prove the two partial Harnack inequalities. 
The proof of these inequalities is based on a comparison with suitable test functions. In order to build these "barriers", we will use the following function $\varphi$. Let $Q=(0, \cdots, 0, \frac{1}{5})$ and define $\varphi: B_1 \to \mathbb{R}$ by
\begin{equation}
\label{E14}
\varphi(x)=
\begin{cases}
1, \qquad & \mathrm{if} \,\,\, x \in B_{\frac{1}{100}}(Q), \\
\kappa \left( |x-Q|^{-\gamma}-(\frac{3}{4})^{-\gamma} \right), \qquad & \mathrm{if} \,\,\, x \in B_{\frac{3}{4}}(Q) \setminus \overline{B_{\frac{1}{100}}(Q)}, \\
0, \qquad & \mathrm{otherwise},
\end{cases}
\end{equation}
where 
$$ \gamma=\gamma(n,p):=\max \left\{1,\frac{1+n-p}{p-1}, \frac{1+n}{p-1}-2, n+p-3 \right\}, $$
and the constant $\kappa$ is chosen in such a way that $\varphi$ is continuous. One can check that $\varphi$ has the following properties:
\begin{enumerate}
\item[($\varphi.1$)] $0 \leq \varphi \leq 1 $ in $\mathbb{R}^n$, and $\varphi=0$ on $\partial B_1$;
\item[($\varphi.2$)] For $s>0$ small,
$$ -\mathrm{div} \left( \left| \textbf{e}_n- s \nabla \varphi \right|^{p-2} \left(  \textbf{e}_n- s \nabla \varphi \right) \right) \geq C(n,p,s)>0, \qquad \text{in} \quad \{\varphi>0\} \setminus \overline{B_{\frac{1}{100}}(Q)}, $$
(with fairly simple computations same as the ones which have been done in \cite[Lemma 4.2]{MR4273843}); 
\item[($\varphi.3$)] $\partial_n \varphi >0$ in $\{\varphi>0\} \cap \{|x_n| \leq \frac{1}{100}\}$;
\item[($\varphi.4$)] $\varphi \geq c(n,p)>0$ in $B_{\frac{1}{6}}$;
\end{enumerate}
where $C(n,p,s)$ and $c(n,p)$ are constants.

\begin{lemma}[Partial Boundary Harnack I]
\label{L8}
Given $1<p<\infty$ and $\lambda_+\ge \lambda_- >0$, there exist constants $\overline{\epsilon}=\overline{\epsilon}(n,\lambda_{\pm},p)>0$ and $\overline{c}=\overline{c}(n, \lambda_{\pm},p) \in (0,1)$ such that, for every function $u : B_4 \to \mathbb{R}$ satisfying $(b)-(c)$ in Theorem \ref{T2}, the following properties hold true.

Let $a_{\pm}, b_{\pm} \in (-\frac{1}{100},\frac{1}{100})$ be such that
$$ b_+ \leq  b_- \leq a_- \leq a_+, $$
and
$$ (a_--b_-)+(a_+-b_+) \leq \overline{\epsilon}. $$
Assume that for $x \in B_4$
$$ \lambda_+ (x_n+b_+)^+ \leq u^+(x) \leq \lambda_+ (x_n+a_+)^+, $$
and
$$ -\lambda_- (x_n+b_-)^- \leq -u^-(x) \leq -\lambda_-(x_n+a_-)^-. $$
Then, one can find new constants $\overline{a}_{\pm}, \overline{b}_{\pm} \in (-\frac{1}{100}, \frac{1}{100})$, with
$$ \overline{b}_+ \le  \overline{b}_- \leq \overline{a}_- \leq \overline{a}_+, $$
and
$$ \overline{a}_--\overline{b}_- \leq \overline{c} (a_--b_-), \qquad \overline{a}_+-\overline{b}_+ \leq \overline{c} (a_+-b_+) $$
such that for $x \in B_{\frac{1}{6}}$
$$ \lambda_+ (x_n+\overline{b}_+)^+ \leq u^+(x) \leq \lambda_+ (x_n+\overline{a}_+)^+, $$
and
$$ -\lambda_- (x_n+\overline{b}_-)^- \leq -u^-(x) \leq - \lambda_- (x_n+\overline{a}_-)^-. $$
\end{lemma}
\begin{remark}
We need to remark that the assumption $\lambda_+\ge \lambda_-$ is not restrictive as one can always replace $u(x)$ by $-u(x', -x_n)$ in $J_{\mathrm{TP}}$. Also, when $\lambda_+\le \lambda_-$ the similar result holds if we replace the order of $a_\pm, b_\pm$ with
$ b_- \le b_+\le a_+ \le a_- .$ 
\end{remark}
\begin{proof}[Proof of Lemma \ref{L8}]
Let us show how to improve the positive part. 
More precisely,  given $a_+,a_-,b_+,b_-$  we will show how we can find $\overline{a}_+$ and $\overline{b}_+$. 
The proof for $\overline{b}_-$ and $\overline{a}_-$ follows in the same way.
We let
$$P=(0, \cdots,0,2), $$
and distinguish two cases:

\medskip
\noindent \textit{\bf Case 1. Improvement from above.} 
 Assume that, at the point $P$, $u^+$ is closer to $\lambda_+(2+b_+)^+$ than to the upper barrier $\lambda_+(2+a_+)^+$. Precisely that
$$ u^+(P) \leq \lambda_+ (2+a_+)^+-\frac{\lambda_+ (a_+-b_+)}{2}. $$
In this case, we will show that $u(x)$ is less than $\lambda_+ (x_n+\overline{a}_+)^+$ in a smaller ball centered at the origin for $\overline{a}_+$ strictly smaller than $a_+$.

We start by setting
$$ \epsilon:=a_+-b_+ \leq \overline{\epsilon}. $$
Then
$$u^+(P)\leq \lambda_+ (2+a_+)^+-\frac{\lambda_+ \epsilon}{2} \leq \lambda_+ (1-c \epsilon)(2+a_+)^+, $$
for a suitable (universal) constant $c$. We can thus apply (the scaled version of) Lemma \ref{L7} to $u^+$, to infer the existence of a constant $\tau=\tau(n,p)$ such that
\begin{equation}
\label{EE5}
u^+(x) \leq \lambda_+ (1-\tau \epsilon)(x_n+a_+)^+, \qquad \text{in} \quad B_1.
\end{equation}
For $\varphi$ as in \eqref{E14} and $ t \in [0,1]$, we set 
$$ f_t= \lambda_+ \left(1-\tau \frac{\epsilon}{2} \right)(x_n+a_+-tc \epsilon \varphi )^+, $$
where $c=c(n,p)$ is a small constant chosen such that for all $x \in B_{\frac{1}{100}}(Q)$ and $t \in [0,1)$,
\begin{equation}
\label{EE6}
\begin{aligned}
u(x) &\leq \lambda_+ (1-\tau \epsilon)(x_n+a_+)^+ \\
& \leq \lambda_+ \left(1-\tau \frac{\epsilon}{2} \right)(x_n+a_+-c\epsilon)^+ < f_t(x), 
\end{aligned}
\end{equation}
where we have used that $(x_n+a_+)$ is within two universal constant for $x \in B_{\frac{1}{100}}(Q)$. 

We now let $\overline{t} \in (0,1]$ the largest $t$ such that $f_t \geq u $ in $B_1$ and we claim that $\overline{t}=1$. Indeed assume that $\overline{t}<1$, then there exists $\overline{x} \in B_1$ such that
\begin{equation}
\label{EE7}
u(x)-f_{\overline{t}}(x) \leq u(\overline{x})-f_{\overline{t}}(\overline{x})=0, \qquad \text{for all \,\,\, $x \in B_1$}.
\end{equation}
Note that by \eqref{EE6}, $\overline{x} \not \in B_{\frac{1}{100}}(Q)$, while, by  \eqref{EE5}, $\overline{x} \in \{\varphi>0\}$. 
Moreover, if $u(\overline{x}) = f_{\overline{t}}(\overline{x})>0$, by $(\varphi.2)$ we will have
$$\Delta_p f_{\overline{t}}(\overline{x})=\left( \lambda_+ \left(1-\tau \frac{\epsilon}{2} \right) \right)^{p-1} \mathrm{div} \left( \left|\textbf{e}_n-\overline{t}c\epsilon \nabla \varphi (\overline{x}) \right|^{p-2} \left(\textbf{e}_n-\overline{t}c\epsilon \nabla \varphi (\overline{x}) \right) \right) <0, $$
but, since $\Delta_p u(\overline{x})=0$, we reach a contradiction with \eqref{EE7} and the definition of viscosity solution for the $p$-harmonic function $u$.  
Hence, $u(\overline{x}) = f_{\overline{t}}(\overline{x})=0$. 
Now recall the free boundary condition  \eqref{OVERDETERMINED}  and apply $(\varphi.3)$ to get 
$$ \lambda_+^p \leq |\nabla f_{\overline{t}}(\overline{x})|^p = \lambda_+^p \left(1-\tau \frac{\epsilon}{2} \right)^p \left(1-p c  \overline{t} \epsilon \partial_n \varphi(\overline{x})+ O(\epsilon^2)\right) < \lambda_+^p, $$
provided $\epsilon \leq \overline{\epsilon}(n,\lambda_+,p) \ll 1$ (note that necessarily $u(\overline{x})=0$ and thus $\overline{x} \in \{|x_n|\leq \frac{1}{100}\}$). 
This contradiction implies that $\overline{t}=1$. 
Hence, by $(\varphi.4)$, we get for all $x \in B_{\frac{1}{6}}$
$$ u(x) \leq \lambda_+ \left(1-\tau \frac{\epsilon}{2} \right) (x_n+a_+-c \epsilon \varphi)^+ \leq \lambda_+ (x_n+a_+-\overline{c} \epsilon)^+, $$
for a suitable constant $\overline{c}=\overline{c}(n,p)$. Setting
$$ \overline{a}_+=a_+-\overline{c} \epsilon, \qquad \overline{b}_+=b_+, $$
and recalling that $\epsilon=a_+-b_+$ we finish the proof in this case.

\medskip
\noindent \textit{\bf Case 2. Improvement from below.} We now assume that, at  point $P$, $u^+$ is closer to $\lambda_+(2+a_+)^+$ than to $\lambda_+(2+b_+)^+$. Hence, we have
$$ u^+(P) \geq \lambda_+(2+b_+)^++\frac{\lambda_+(a_+-b_+)}{2}, $$ 
and we set again
$$ \epsilon:=a_+-b_+ \leq \overline{\epsilon}. $$
Arguing as in Case 1, by Lemma \ref{L7}, there exists a constant $\tau=\tau(n,p)$ such that 
\begin{equation}
\label{EE8}
u^+(x) \geq \lambda_+ (1+\tau \epsilon)(x_n+b_+)^+, \qquad \text{in} \quad B_1.
\end{equation}
We need now to distinguish two further sub-cases:\\
\textit{\bf Case 2.1:} Suppose that 
$$ \eta \epsilon \leq b_--b_+, $$
where $\eta \ll \tau$ is a small universal constant which we will choose at the end of the proof. In this case, for $x \in B_1$,
\begin{equation}
\label{EE9}
\begin{aligned}
u(x) &\geq \lambda_+ (1+\tau \epsilon)(x_n+b_+)^+-\lambda_-(x_n+b_-)^- \\
& \geq \lambda_+ (1+\tau \epsilon)(x_n+b_+)^+-\lambda_-(1-c_1\eta \epsilon)(x_n+b_+)^-,
\end{aligned}
\end{equation}
for a suitable universal constant $c_1$. We now take $\varphi$ as in \eqref{E14} and set, for $t \in [0,1]$,
$$ f_t(x)=\lambda_+ \left(1+\tau \frac{\epsilon}{2} \right)(x_n+b_++c_2 t \epsilon \varphi)^+-\lambda_-(1-c_1\eta \epsilon)(x_n+b_++c_2t \epsilon \varphi)^-, $$
for a suitably small universal constant $0 < c_2 \ll \tau$, chosen so that for all $x \in B_{\frac{1}{100}}(Q)$
$$ (1+\tau \epsilon)(x_n+b_+)^+ \geq \left(1+\tau \frac{\epsilon}{2} \right) (x_n+b_++c_2 \epsilon)^+. $$
This together with \eqref{EE8} implies that 
\begin{equation}
\label{EE10}
\begin{aligned}
u(x) \geq \lambda_+ (1+\tau \epsilon)(x_n+b_+)^+ & \geq \lambda_+ \left(1+\tau \frac{\epsilon}{2} \right)(x_n+b_++c_2 \epsilon)^+ \\
& \geq f_1(x) \geq f_t(x), \quad \text{for all \,\, $x \in B_{\frac{1}{100}}(Q), t \in [0,1]$}. 
\end{aligned}
\end{equation}
Furthermore $u \geq f_0$ in $B_1$ thanks to \eqref{EE9}.
Similar to Case 1, let $\overline{t}$ be the biggest $t$ such that $f_t \leq u$ in $B_1$ and $\overline{x}$ be the first contact point, so that
$$ u(x)-f_{\overline{t}}(x) \geq u(\overline{x})-f_{\overline{t}}(\overline{x})=0, \qquad \text{for all \,\,\,$x \in B_1$.} $$
Since, by using ($\varphi.2$), it can be checked that 
$$ \Delta_p f_{\overline{t}} >0, \qquad \text{on}\quad \{f_t \neq 0 \} \setminus B_{\frac{1}{100}}(Q), $$  
therefore, as in Case 1, $\overline{x}$ is a free boundary point. Moreover, since $f_{\overline{t}}$ changes sign in a neighborhood of $\overline{x}$:
$$
\begin{array}{rr}
\mathrm{either} & \quad \overline{x} \in \Gamma^+_{\mathrm{OP}}=\partial \Omega^+_u \setminus \partial \Omega^-_u, \\
\\
\mathrm{or} & \quad \overline{x} \in \Gamma_{\mathrm{TP}}=\partial \Omega^+_u \cap \partial \Omega^-_u.
\end{array}
$$
In the first case, by definition of viscosity solution and $(\varphi.3)$,
$$ \lambda_+^p \geq |\nabla f^+_{\overline{t}}(\overline{x})|^p = \lambda_+^p \left(1+\tau \frac{\epsilon}{2} \right)^p \left(1+p c_2  \overline{t} \epsilon \partial_n \varphi(\overline{x})+ O(\epsilon^2)\right) > \lambda_+^p, $$
a contradiction for $\epsilon \ll 1 $. In the second case, we have a contradiction as well, since (recall also the assumption $\lambda_+-\lambda_-\ge0$)
$$
\begin{aligned}
\lambda_+^p-\lambda_-^p & \geq |\nabla f^+_{\overline{t}}|^p-|\nabla f^-_{\overline{t}}|^p \\
&= \left(\lambda_+^p \left( 1+ \tau \frac{\epsilon}{2} \right)^p - \lambda_-^p (1-c_1\eta \epsilon)^p \right) \left(1+p c_2  \overline{t} \epsilon \partial_n \varphi(\overline{x})+ O(\epsilon^2)\right) \\
& > \lambda_+^p - \lambda_-^p, 
\end{aligned}
$$
provided $\epsilon \ll 1 $ (only depending on $n$, $\lambda_+$ and $p$). Hence, $\overline{t}=1$, $u \geq f_1$ (so $u^+ \ge f_1^+$) which implies the desired conclusion by setting
$$ \overline{a}_+=a_+, \qquad \overline{b}_+=b_++\overline{c}_2 \epsilon, $$
for a suitable constant $\overline{c}_2=\overline{c}_2(n,p)$ and by recalling that $\epsilon=a_+-b_+$.

\noindent \textit{\bf Case 2.2:} Assume instead that:
$$0\le  b_--b_+ \leq \eta \epsilon, $$
where $\eta=\eta(n,p)$ will be determined later. 
In this case we consider the family of functions $$ f_t(x)=\lambda_+ \left(1+\tau \frac{\epsilon}{2} \right)(x_n+b_++\eta t \epsilon \varphi)^+-\lambda_-(x_n+b_-)^-. $$
Since $\varphi \leq 1$, this function is well defined due to $b_- \leq b_+ + \eta\epsilon$. 
Moreover, $u \geq f_0$ and, thanks to \eqref{EE8} and by assuming $\eta$ is sufficiently small (this can  also be determined universally  depending only on the dimension and $p$) we will have,
$$ u(x) \geq f_1(x) \geq f_t(x), \qquad \text{for all \,\,\, $x \in B_{\frac{1}{100}}(Q), t \in [0,1]$.} $$
We consider again the first touching time $\overline{t}$ and the first touching point $\overline{x}$. 
By arguing as in the previous cases, we get $\overline{x} \in \{u=0\}\cap \{|x_n|\leq \frac{1}{100}\}$.
Also, the definition of $f_t$ yields that $\overline{x} \in \partial\{f_{\overline{t}}>0\}$.
This infer that  $\overline{x} \in \partial \Omega^+_u \setminus \partial \Omega^-_u$ (note that $\varphi(\overline{x})<1$). 
However, again by arguing as in Case 2.1, this is in contradiction with $u$ being a viscosity solution. 
We now conclude as in the previous cases.
\end{proof}

The following lemma addresses the situation in which the origin is not a branching point.

\begin{lemma}[Partial Boundary Harnack II]
\label{L9}
Given $1<p<\infty$ and $0< L_0, L_1$ and assume that $0<\lambda_- \le \lambda_+\le L_1$, then there exist constants $\overline{\epsilon}=\overline{\epsilon}(n,L_0,L_1,p)>0$, $M=M(n,L_0,L_1,p)$ and $c=c(n,L_0,L_1,p) \in (0,1)$ such that for every function $u: B_4 \to \mathbb{R}$ satisfying $(b)-(c)$ in Theorem \ref{T2} the following property holds true. If there are constants $a,b \in (-\frac{1}{100}, \frac{1}{100})$ with
$$ 0 \leq a-b \leq \overline{\epsilon}, $$
such that for $x \in B_4$
$$ H_{\alpha, \textbf{e}_n} (x+b\textbf{e}_n) \leq u(x) \leq H_{\alpha, \textbf{e}_n}(x+a\textbf{e}_n), $$
and
$$ \max(\lambda_+, L_0) + M \epsilon \leq \alpha \leq L_1, $$  
then there are constants $\overline{a}, \overline{b} \in (-\frac{1}{100}, \frac{1}{100})$ with 
$$ 0 \leq \overline{a}-\overline{b} \leq c(a-b), $$
such that for $x \in B_{\frac{1}{6}}$
$$ H_{\alpha, \textbf{e}_n} (x+\overline{b}\textbf{e}_n) \leq u(x) \leq H_{\alpha, \textbf{e}_n}(x+\overline{a}\textbf{e}_n). $$
\end{lemma}

\begin{proof}
We consider the point $P=(0,\cdots,0,2)$ and distinguish two cases (note that one of these inequalities is always satisfied):
$$
\begin{array}{rr}
\mathrm{either} & \quad H_{\alpha, \textbf{e}_n} \left( P+b\textbf{e}_n \right)+\dfrac{\alpha(a-b)}{2} \leq u(P), \\
\\
\mathrm{or} & \quad H_{\alpha, \textbf{e}_n} \left( P+a\textbf{e}_n \right)-\dfrac{\alpha(a-b)}{2} \geq u(P).
\end{array}
$$
Since the argument in both cases is completely symmetric we only consider the second case. 
If we set
$$ \epsilon=a-b, $$
by Lemma \ref{L7} and by arguing as in Lemma \ref{L8} we deduce the existence of a constant $\tau=\tau(n,p)$ such that
$$ u(x) \leq \alpha (1-\tau \epsilon)(x_n+a)^+-\beta(x_n+a)^-, $$
in $B_1$. We let $\varphi$ as in \eqref{E14} and  set 
$$ f_t(x)=\alpha \left( 1- \tau \frac{\epsilon}{2} \right)(x_n+a-ct \epsilon \varphi)^+-\beta(x_n+a-ct \epsilon \varphi)^-, $$
where $c=c(n,p)$ is a constant chosen such that 
$$ u(x) \leq f_1(x) \leq f_t(x), \qquad \text{for all \,\,\, $x \in B_{\frac{1}{100}}(Q), t \in [0,1]$,} $$ 
where, $Q=(0,\cdots,0,\frac{1}{5})$. 
As in Lemma \ref{L8}, we let $\overline{t}$ and $\overline{x}$ be the first contact time and the first contact point and we aim to show that $\overline{t}=1$. For this purpose, we note that, by the same arguments as in Lemma \ref{L8}, necessarily $\overline{x} \in \{u=0\}$. We claim that
$$ \overline{x} \in \Gamma_{\mathrm{TP}}=\partial \Omega^+_{u} \cap \partial \Omega^-_{u}. $$
Indeed, otherwise $\overline{x} \in \partial \Omega^-_u \setminus \partial \Omega^+_u$ (the case $\overline{x} \in \partial \Omega^+_u \setminus \partial \Omega^-_u$ will be impossible since $f_{\overline{t}}$ is negative in a neighborhood of $\overline{x}$). 
And by definition of viscosity solution, this along with \eqref{E0-Con} would imply
\begin{equation*}
\begin{split}
\lambda_-^p \geq |\nabla f^-_{\overline{t}}(\overline{x})|^p & = \beta^p \left(1 - p c  \overline{t} \epsilon \partial_n \varphi(\overline{x})+O(\epsilon^2) \right)  \\
 & \geq \left(\lambda_-^p- \lambda_+^p+\alpha^p \right) \left(1 - p c  \overline{t} \epsilon \partial_n \varphi(\overline{x})+O(\epsilon^2) \right)\\
 & \ge \left (\lambda_-^p- \lambda_+^p+ \left(\max(\lambda_+,L_0)+M\epsilon \right)^p \right) \left(1 - p c  \overline{t} \epsilon \partial_n \varphi(\overline{x})+O(\epsilon^2) \right)\\
 & = \lambda_-^p +p \left(L_0^{p-1} M-c  \overline{t}  \partial_n \varphi(\overline{x}) \right)\epsilon + O(\epsilon^2),
\end{split}
\end{equation*}
where the implicit constants in $O(\epsilon^2)$ can control by  $L_1$, $p$ and $n$.
This inequality is impossible if $M$ is chosen sufficiently large. 

Hence $\overline{x} \in \partial \Omega^+_u \cap \partial \Omega^-_u$. This however implies:
$$
\begin{aligned}
\lambda_+^p-\lambda_-^p &\leq |\nabla f^+_{\overline{t}}(\overline{x})|^p - |\nabla f^-_{\overline{t}}(\overline{x})|^p \\
& = \left( \alpha^p \left( 1- \tau \frac{\epsilon}{2} \right)^p - \beta^p \right) \left(1 - pc \overline{t} \epsilon \partial_n \varphi(\overline{x})+O(\epsilon^2) \right) \\
& <\alpha^p-\beta^p = \lambda_+^p - \lambda_-^p,
\end{aligned}
$$
provided $\overline{\epsilon}$ and as a consequence of $\epsilon=a-b \leq \overline{\epsilon}$, $\epsilon$ is chosen small enough, where we have used $(\varphi.3)$ and the equality 
$$0 \le \lambda_+^p-\lambda_-^p=\alpha^p-\beta^p. $$
This contradiction shows that $\overline{t}=1$ and as in Lemma \ref{L8}, this completes the proof.
\end{proof}

With Lemmas \ref{L7} and \ref{L8} at hand the proof of Lemma \ref{L5.5} is as follows.

\begin{proof}[Proof of Lemma \ref{L5.5}] We distinguish two cases:

\medskip
\noindent\underline{\textbf{Case $0\leq \ell < +\infty$}}: 
By triangular inequality we have  
$$ \|u_k - H_{\lambda_+,\textbf{e}_n}\|_{L^{\infty}(B_1)} \leq \epsilon_k \left(1+ 2 \ell \max(\lambda_+^{1-p}, \lambda_-^{1-p}) \right), $$
for $k$ sufficiently large. Define the bounded sequence $w_k$ by 
$$ 
w_k(x)=
\begin{cases}
w_{+,k}(x):=\dfrac{u_k(x)-\lambda_+ x_n^+}{\alpha_k \epsilon_k}, \qquad & x \in \Omega_{u_k}^+ \cap B_1, \\
\\
w_{-,k}(x):=\dfrac{u_k(x)+\lambda_-x_n^-}{\beta_k \epsilon_k}, \qquad & x \in \Omega_{u_k}^- \cap B_1.
\end{cases}
$$
Now we can repeatedly apply Lemma \ref{L8} to deduce that $w_k$ satisfies 
\begin{equation}\label{Eq26}
|w_k(x)-w_k(y)| \le C|x-y|^\gamma, \qquad \text{when }x,y\in B_{\frac12}, \text{ and } |x-y| \ge \frac{\epsilon_k}{\overline{\epsilon}},
\end{equation}
for some universal exponent $0<\gamma<1$ and constant $C$; see \cite[Corollary 4.2]{MR3218810}.
This gives that the graphs of 
$$ \tilde{\Gamma}_k^{\pm}:= \{ (x,w_{\pm,k}(x)) \, : \, x \in \overline{\Omega_{u_k}^{\pm} \cap B_{\frac{1}{2}}} \}, $$
converge, up to a subsequence, in the Hausdorff distance to the closed graphs
$$ \tilde{\Gamma}_{\pm}:= \{ (x,w_{\pm}(x)) \, : \, x \in \overline{B^{\pm}_{\frac{1}{2}}} \}, $$
where $w \in C^{0,\alpha}$ for some $\alpha>0$. Since 
$$ h_k(x):=\frac{H_{\alpha_k,\textbf{e}_n}-H_{\lambda_+,\textbf{e}_n}}{\epsilon_k} \to 
\begin{cases}
\lambda_+^{1-p} \ell x_n, \qquad & x_n>0, \\
\\
\lambda_-^{1-p} \ell x_n, \qquad & x_n<0,
\end{cases}
$$
the original sequence $v_k$ satisfies that their graphs,
converges to the graph of a limiting function $v$ as we wanted, this in particular proves (i), (ii), and (iii).

Since $0 \in \partial \Omega_{u_k}^+ \cap \partial \Omega_{u_k}^-$ then $0$ is in the domain of $v_{\pm,k}$ and
$$ v_{\pm,k}(0)=0, $$
which implies that $v_{\pm}(0)=0$. 
To show that $v_+(x) \leq v_-(x)$ for $x=(x',0) \in  B_{\frac{1}{2}}$, we simply exploit (iii) at the points $x_k^{\pm}=(x',t_k^{\pm})$ where
$$ t_k^+=\sup\{t \, : \, (x',t) \in \partial \Omega^+_{u_k}\} \qquad \text{and} \qquad  t_k^-=\inf\{t \, : \, (x',t) \in \partial \Omega^-_{u_k}\}, $$
and by noticing that $ t_k^- \le t^+_k$. 

Finally, to see the last claim, \eqref{E10}, it is enough to note that if $x_k \in \partial \Omega_{u_k}^+ \cap \partial \Omega^-_{u_k}$ is converging to $x$ then $v_{+,k}(x_k)=v_{-,k}(x_k)$ and thus $v_+(x)=v_-(x)$, yielding $x \in \mathcal{C}$.

\medskip
\noindent\underline{\textbf{Case $\ell =\infty$}}:
In this case, the conclusion follows exactly with a similar argument by using repeatedly Lemma \ref{L9} for function $v_k$ to obtain a relation similar to \eqref{Eq26} for functions $v_k$.
\end{proof}

\subsection{The linearized problem: proof of Lemma \ref{L6}.}
\label{3-2}
Lemma \ref{L6} proves through the following technical lemma, whose proof is easily obtained by adapting the one in \cite[Lemma 3.10]{MR4285137} exactly.
 Then we present the statement without proof. 
 
\begin{lemma}
\label{L12}
Let $u_k$, $\epsilon_k$ and $\alpha_k$ be as in the statement of Lemma \ref{L5.5}, $v_k$ be defined by \eqref{E9} and $v_{\pm}$ be as in Lemma \ref{L5.5}. Then:
\begin{enumerate}
\item [(1)] Let $P_+$ be a quadratic polynomial with $\mathcal{L}_p(P_+) > 0$ (or $\mathcal{L}_p(P_+) < 0$) on $B^+_{\frac{1}{2}}$ touching $v_+$ strictly from below (above) at a point $x_0 \in \{x_n=0\}\cap B_{\frac{1}{2}}$. Then, there exists a sequence of points $\partial \Omega^+_{u_k} \ni x_k \to x_0$ and a sequence of comparison functions $Q_k$ such that $Q_k$ touches from below (above) $u_k^+$ at $x_k$, and such that 
\begin{equation}
\label{E18}
\nabla Q^+_k(x_k)=\alpha_k \textbf{e}_n+ \epsilon_k\alpha_k \nabla P_+(x_0)+ o(\epsilon_k).
\end{equation}

\item [(2)] Let $P_-$ be a quadratic polynomial with $\mathcal{L}_p(P_-) > 0$ ($\mathcal{L}_p(P_-) < 0$) on $B^-_{\frac{1}{2}}$ touching $v_-$ strictly from below (above) at a point $x_0 \in \{x_n=0\}\cap B_{\frac{1}{2}}$. Then, there exists a sequence of points $\partial \Omega^-_{u_k} \ni x_k \to x_0$ and a sequence of comparison functions $Q_k$ such that $Q_k$ touches from below (above) $-u_k^-$ at $x_k$, and such that 
\begin{equation}
\label{E19}
\nabla Q^-_k(x_k)=-\beta_k \textbf{e}_n+ \epsilon_k\beta_k \nabla P_-(x_0)+ o(\epsilon_k).
\end{equation}

\item [(3)] Let $s,t \in \mathbb{R}$ and $\tilde{P}$ be a quadratic polynomial on $B_{\frac{1}{2}}$ such that $\partial_n \tilde{P} = 0$. Suppose that $\mathcal{L}_p(\tilde{P}) \geq 0$ ($\mathcal{L}_p(\tilde{P}) \leq 0$) and that the function 
$$ P:=sx_n^+-tx_n^-+\tilde{P}, $$
touches $v$ strictly from below (above) at a point $x_0 \in \mathcal{C}$. Then, there exists a sequence of points $x_k \to x_0$ and a sequence of comparison functions $Q_k$ such that $Q_k$ touches from below (above) the function $u_k$ at $x_k \in \partial \Omega_{u_k}$, and such that 
\begin{equation}
\label{E20}
\begin{aligned}
& \nabla Q^+_k(x_k)=\alpha_k(1 + \epsilon_k s)  \textbf{e}_n+ o(\epsilon_k), \\
& \nabla Q^-_k(x_k)=-\beta_k (1+ \epsilon_k t) \textbf{e}_n+ o(\epsilon_k).
\end{aligned}
\end{equation}
In particular, if $s>0$ and $Q_k$ touches $u_k$ from below then $x_k \not\in \partial \Omega^-_{u_k} \setminus \partial \Omega^+_{u_k}$, while if $t<0$ and $Q_k$ touches $u_k$ from above then $x_k \not\in \partial \Omega^+_{u_k} \setminus \partial \Omega^-_{u_k}$.
\end{enumerate}
\end{lemma}

\begin{proof}[Proof of Lemma \normalfont{\ref{L6}}]
{\bf Step 1:} In this step, we prove $\mathcal{L}_p(v_{\pm})=0$ in $B^{\pm}_{\frac{1}{2}}$.\\
Let $P(x)$ be a quadratic polynomial touching $v=v_+$ at $\overline{x} \in B^{+}_{\frac{1}{2}}$ strictly from below. We need to show that at this point
$$ \mathcal{L}_p(P)=\Delta P + (p-2) \partial_{nn} P \leq 0. $$
Since $v_{+,k} \to v_+$, there exist points $x_k \in \Omega^+_{u_k} \cap B_{\frac{1}{2}}$, $x_k \to \overline{x}$ and constants $c_k \to 0$ such that
\begin{equation}
\label{EEE1}
v_{+,k}(x_k)=P(x_k)+c_k,
\end{equation}
 and
 \begin{equation}
\label{EEE2}
v_{+,k} \geq P+c_k, \qquad \text{in a neighborhood of $x_k$.} 
\end{equation}
From the definition of $v_{+,k}$, \eqref{EEE1} and \eqref{EEE2} read
$$ u_k(x_k)=Q_k(x_k), $$
and
$$ u_k(x) \geq Q_k(x), \qquad \text{in a neighborhood of $x_k$,} $$
where
$$ Q_k(x)=\epsilon_k  \alpha_k (P(x)+c_k)+\alpha_k x_n^+. $$
Note that
 \begin{equation}
\label{EEE3}
\nabla Q_k = \epsilon_k  \alpha_k  \nabla P + \alpha_k \textbf{e}_n,
\end{equation}
thus,
 \begin{equation}
\label{EEE4}
\nabla Q_k(x_k) \neq 0, \qquad \text{for $k$ large.}
\end{equation}

Since $u_k$ is $p$-harmonic and $Q_k$ touches $u_k$ from below at $x_k$, and $\nabla Q_k(x_k) \neq 0$, by the equivalence of weak and viscosity solutions of $p$-harmonic functions,  we get
$$ 
\begin{aligned}
0 &\geq \Delta_p Q_k(x_k) \\
& = \mathrm{div}\left( |\nabla Q_k(x_k)|^{p-2} \nabla Q_k(x_k) \right) \\
& = \left| \nabla Q_k(x_k) \right|^{p-2} \Delta Q_k (x_k) + (p-2) \left| \nabla Q_k(x_k) \right|^{p-4} \sum_{i,j=1}^{n} {Q_k}_{x_i}(x_k) {Q_k}_{x_j}(x_k) {Q_k}_{x_i x_j}(x_k) \\
& = \epsilon_k  \left| \nabla Q_k(x_k) \right|^{p-2} \Delta P (x_k) + \epsilon_k  (p-2) \left| \nabla Q_k(x_k) \right|^{p-4} \sum_{i,j=1}^{n} {Q_k}_{x_i}(x_k) {Q_k}_{x_j}(x_k) {P}_{x_i x_j}(x_k).
\end{aligned}
$$
Now, dividing both sides by $\epsilon_k$, and passing to the limit $k \to \infty$, and recalling that 
$$ \nabla Q_k(x_k) \to \alpha_{\infty} \textbf{e}_n, $$
we conclude that 
$$ \Delta P(\overline{x}) + (p-2) \partial_{nn} P(\overline{x}) \leq 0. $$ 
Touching from above and reaching the opposite inequality is similar. Also, the reasoning of the case $v=v_-$ in the negative half ball $B^-_{\frac{1}{2}}$ can be done similarly.

\medskip
{\bf Step 2:} In this step, we show that $\mathcal{J} = \emptyset$, when $\ell=\infty$.\\
Assume the contrary, since the set $\{v_->v_+\}$ is open in $\{x_n=0\}$, it contains a $(n-1)$-dimensional ball
$$ B'_{\epsilon}(y'):=B_{\epsilon}((y',0)) \cap \{x_n=0\} \subset \mathcal{J}. $$
Next, let $P$ be the polynomial
$$ P(x)=A \left(n-\frac{1}{2} \right)x_n^2-|x'-y'|^2 -Bx_n, \qquad \text{where} \quad x=(x',x_n), $$
for some constants $A,B$. We first choose suitable $A = A(p)$  so that $\mathcal{L}_p(P)>0$. Notice that
$$ P<v_+ \qquad \text{on} \quad \{|x'-y'|=\epsilon\} \cap \{x_n=0\}. $$
Moreover, we choose $ B \gg A $ so that
$$ P < v_+ \qquad \text{on} \quad B_{\epsilon}((y',0)). $$ 
Now we can translate $P$ first down and then up to find that there exists $C$ such that $P+C$ is touching $v_+$ from below at a point $x_0 \in B_{\epsilon}((y',0)) \cap \{x_n \geq 0\}$. Since $\mathcal{L}_p(P)>0$, the touching point can not be in the interior of the (half) ball, and thus $x_0 \in B'_{\epsilon}(y') \subset \mathcal{J}$. 

By using Lemma \ref{L12}, there exists a sequence of points $\partial \Omega^+_{u_k} \ni x_k \to x_0$ and of functions $Q_k$ touching $u_k^+$ from below at $x_k$ such that 
$$ \nabla Q^+_k(x_k)=\alpha_k \textbf{e}_n + \epsilon_k  \alpha_k \nabla P(x_0) + o(\epsilon_k). $$
Since $x_0 \in \mathcal{J}$, by \eqref{E10} in Lemma \ref{L5.5}, $x_k \in \partial \Omega^+_{u_k} \setminus \partial \Omega^-_{u_k}$. Hence, by (ii) in Lemma \ref{L3}
$$ \lambda_+^p \geq |\nabla Q_k^+(x_k)|^p \geq \alpha^p_k+p\alpha_k^{p} \epsilon_k \partial_n P(x_0)+o(\epsilon_k). $$
Now recalling \eqref{E8}, the definition of $\ell$, 
$$ -B=\partial_n P(x_0) \leq \frac{\lambda_+^p-\alpha_k^p}{p\alpha_k^{p} \epsilon_k} + o(1) \to - \infty. $$
This contradiction proves that $\mathcal{J}=\emptyset$.

\medskip
{\bf Step 3:} In this step, we check  the transmission condition in \eqref{E12}  when $\ell=\infty$.\\
Let us show that
$$ \alpha_{\infty}^{p} \partial_n v_+ - \beta_{\infty}^{p} \partial_n v_- \leq 0, $$
the opposite inequality can then be proved in a similar way. 
Suppose that there exist $s$ and $t$ with $\alpha_{\infty}^{p} s> \beta_{\infty}^{p} t$ and a polynomial $\tilde{P}$ with $\mathcal{L}_p(\tilde{P})>0$ and $\partial_n \tilde{P}=0$ such that
$$ P=sx_n^+-tx_n^-+\tilde{P}, $$
touches $v$ strictly from below at a point $x_0 \in \{x_n=0\} \cap B_{\frac{1}{2}}$ (note that $\{x_n=0\} \cap B_{\frac{1}{2}}=\mathcal{C}$ due to the previous step and Lemma \ref{L5.5}). 
By Lemma \ref{L12} there exists a sequence of points $\partial \Omega^+_{u_k} \cup \partial \Omega^-_{u_k} \ni x_k \to x_0$ and a sequence of comparison functions $Q_k$ touching $u_k$ from below at $x_k$ and satisfying \eqref{E20}. In particular, $x_k \not\in \partial \Omega_{u_k}^- \setminus \partial \Omega^+_{u_k}$. We claim that $x_k \in \partial \Omega_{u_k}^+ \cap \partial \Omega_{u_k}^-$. Indeed, otherwise by $(A.1)$ in Lemma \ref{L2},
$$ \lambda_+^p \geq |\nabla Q^+_k(x_k)|^p, $$
and, by arguing as Step 2, this contradicts $\ell=+\infty$. Hence, by $(A.3)$ in Lemma \ref{L2}
$$
\begin{aligned}
\lambda_+^p-\lambda_-^p &\geq |\nabla Q^+_k(x_k)|^p-|\nabla Q^-_k(x_k)|^p \\
& = \alpha^p_k-\beta_k^p+p\epsilon_k \left(\alpha_k^{p}s - \beta_k^{p} t \right)+o(\epsilon_k) \\
& = \lambda_+^p - \lambda_-^p + p\epsilon_k \left(\alpha_k^{p} s - \beta_k^{p} t \right) + o(\epsilon_k).
\end{aligned}
$$
Dividing by $\epsilon_k$ and letting $k \to \infty$, we obtain the desired contradiction.

\medskip
{\bf Step 4:} Here, we show that $\lambda_{\pm}^{p} \partial_n v_{\pm} \geq - \ell $ on $B_{\frac{1}{2}} \cap \{x_n=0\}$, when $0\leq \ell < \infty$.\\ 
 We focus on $v_-$ since the argument is symmetric. Let us assume that there exists $t \in \mathbb{R}$ with $\lambda_-^{p} t < - \ell $ and a polynomial $\tilde{P}$ with $\mathcal{L}_p(\tilde{P})>0$ and $\partial_n \tilde{P}=0$ such that function
$$P=tx_n+\tilde{P}=tx_n^+-tx_n^-+\tilde{P}, $$
touches $v_-$ strictly from below at a point $x_0 \in \{x_n=0\} \cap B_{\frac{1}{2}}$. Let now $x_k$ and $Q_k$ be as in Lemma \ref{L12}-$(2)$. By optimality conditions 
$$ \lambda_-^p \leq |\nabla Q_k^-(x_k)|^p=\beta_k^p+p\epsilon_k \beta_k^{p} t + o(\epsilon_k). $$
Since $\ell < \infty$, we have $\beta_k= \lambda_-+O(\epsilon_k)$ and so the above inequality leads to 
$$ -\frac{\ell}{\lambda_-^{p}} = \lim_{k \to \infty} \frac{\lambda_-^p-\beta_k^p}{p\epsilon_k \beta_k^{p}} \leq t < -\frac{\ell}{\lambda_-^{p}}, $$
which is a contradiction.

\medskip
{\bf Step 5:}
We now show that $\lambda_{\pm}^{p} \partial_n v_{\pm} = -\ell$ on $\mathcal{J}$, when $0\leq \ell < \infty$.\\ 
By the previous step, it is enough to show that if there exists a polynomial $\tilde{P}$ with $\mathcal{L}_p(\tilde{P})<0$ and $\partial_n \tilde{P}=0$ such that
$$ P=tx_n+\tilde{P}=tx_n^+-tx_n^-+\tilde{P}, $$
touches $v_-$ strictly from above at a point $x_0 \in \mathcal{J}$, then $\lambda_-^{p} t \leq -\ell $. Again, by Lemma \ref{L12}, we find points $x_k \to x_0$ and functions $Q_k$ satisfying \eqref{E19} and touching $-u_k^-$ from below at $x_k$. Since $x_0 \in \mathcal{J}$, by \eqref{E10} in Lemma \ref{L5.5}, $x_k \in \partial \Omega_{u_k}^- \setminus \partial \Omega^+_{u_k}$. Hence, by Lemma \ref{L2},
$$ \lambda_-^p \geq |\nabla Q^-_k(x_k)|^{p} = \beta_k^p+p\beta_k^{p} \epsilon_k t + o(\epsilon_k), $$
which by arguing as above implies that $\lambda_{-}^{p} t \leq -\ell $.

\medskip
{\bf Step 6:} In the last step, we show the transmission condition in \eqref{E13} at points in $\mathcal{C}$.\\
Again by the symmetry of the arguments, we will only show that 
$$ \lambda_+^{p} \partial_n v_+ - \lambda_-^{p} \partial_n v_- \leq 0, \qquad \text{on \,\,\, $\mathcal{C}$}. $$
Let us hence assume that there exist $s$ and $t$ with $\lambda_+^{p} s > \lambda_-^{p} t$ and a polynomial $\tilde{P}$ with $\mathcal{L}_p(\tilde{P})>0$ and $\partial_n \tilde{P}=0$ such that
$$ P = sx_n^+-tx_n^-+\tilde{P}, $$
touches $v_+$ and $v_-$ strictly from below at $x_0 \in \mathcal{C}$. By Lemma \ref{L12}, we find points $x_k \to x_0$ and functions $Q_k$ satisfying \eqref{E20}. In particular $x_k \not \in \partial \Omega_{u_k}^- \setminus \partial \Omega^+_{u_k}$. 
By the previous step we know that $\lambda_-^{p} t \geq -\ell $ and thus $\lambda_+^{p} s > -\ell$, since we are assuming $\lambda_+^{p} s > \lambda_-^{p} t \geq 0$. We now distinguish two cases:
\begin{enumerate}
\item [1)]
$x_k$ is one-phase point, namely $x_k \in \partial \Omega_{u_k}^+ \setminus \partial \Omega_{u_k}^-$. In this case
$$ \lambda_+^p \geq |\nabla Q_k^+(x_k)|^p = \alpha_k^p+p\alpha_k^{p} \epsilon_k s + o(\epsilon_k), $$
which implies that
$$ \lambda_+^{p} s + \ell = \lambda_+^{p} \lim_{k \to \infty} \left( s+\frac{\alpha_k^p-\lambda_+^p}{p\alpha_k^{p} \epsilon_k} \right) \leq 0, $$
in contradiction with $\lambda_+^{p} s > -\ell $.

\item [2)]
$x_k$ is two-phase point, namely $x_k \in \partial \Omega_{u_k}^+ \cap \partial \Omega_{u_k}^-$. Arguing as in Case $1)$, we have that, by Lemma \ref{L2},
$$
\begin{aligned}
\lambda_+^p-\lambda_-^p &\geq |\nabla Q^+_k(x_k)|^p-|\nabla Q^-_k(x_k)|^p \\
& = \alpha^p_k-\beta_k^p+p\epsilon_k \left(\alpha_k^{p}s-\beta_k^{p}t \right)+o(\epsilon_k) \\
& = \lambda_+^p - \lambda_-^p + p \epsilon_k \left(\lambda_+^{p} s - \lambda_-^{p}t \right) + o(\epsilon_k),
\end{aligned}
$$
which gives a contradiction with $\lambda_+^{p} s > \lambda_-^{p} t$, as $\epsilon_k \to 0$. \qedhere
\end{enumerate}
\end{proof}

\subsection{Proof of Lemmas \ref{L4} and \ref{L5}}
\label{3-3}
We recall the following regularity results for the limiting problems.

\begin{lemma}[Regularity for the transmission problem]
\label{LL1}
There exists a universal constant $C=C(\alpha_{\infty},\beta_{\infty},n,p)>0$ such that if $v \in C^0(B_{\frac{1}{2}})$ is a viscosity solution of \eqref{E12} with $\|v\|_{L^{\infty}(B_{\frac{1}{2}})} \leq 1$, then there exist ${\bf  v} \in \mathbb{R}^{n-1}$, $s,t \in \mathbb{R}$ with $\alpha_{\infty}^{p} s=\beta_{\infty}^{p} t $, such that
\begin{equation*}
\sup_{x \in B_r} \left |v(x)-v(0)- \left({\bf  v} \cdot x'+sx_n^+-tx_n^- \right)\right| \leq Cr^2, \qquad \text{ for every }r\le \frac14.
\end{equation*}
\end{lemma}

\begin{proof}
For the proof when $p=2$, we refer to \cite[Theorem 3.2]{MR3218810}. 
This result can be extended easily to the general case (for any $p$) by changing the coordinate such that the operator $\mathcal{L}_p=\Delta + (p-2)\partial_{nn}$ transfer to the Laplacian.
\end{proof}

\begin{lemma}[Regularity for the two-membrane problem]
\label{LL2}
There exists a universal constant $C=C(\lambda_{\pm},n,p)>0$ such that if $v $ is a viscosity solution of \eqref{E13} with $\|v\|_{L^{\infty}(B_{\frac{1}{2}})} \leq 1$, then there exist ${\bf  v} \in \mathbb{R}^{n-1}$, $s, t \in \mathbb{R}$ with $\lambda_{+}^{p} s=\lambda_{-}^{p} t \geq - \ell$, such that
\begin{equation*}
\sup_{x \in B_r^\pm} \left|v(x)-v(0)- \left({\bf  v} \cdot x' + s x_n^+ - t x_n^- \right)\right| \leq C (1+\ell)r^{\frac{3}{2}}, \qquad \text{ for every }r\le  r_p,
\end{equation*}
where $r_p=\frac{1}{4}$ for $1<p\le 2$ and $r_p=\frac{1}{4\sqrt{p-1}}$ for $p>2$.
\end{lemma}

The proof of this lemma can be found in  \cite[Lemma 3.12]{MR4285137} with a minor changes.  
To keep the paper self-contained we will provide a complete proof  for our case in Appendix \ref{appendix}.

Now, the proof of Lemmas \ref{L4} and \ref{L5} by the regularity theory for the limiting problems and a classical compactness argument is available:

\begin{proof}[Proof of Lemma \normalfont{\ref{L4}}]
Toward a contradiction assume that for fixed $\gamma \in (0,\frac{1}{2})$ and $M$, we have a sequences of functions $u_k$ and numbers $\alpha_k$ such that
$$ \epsilon_k=\|u_k- H_{\alpha_k,\textbf{e}_n}\|_{L^{\infty}(B_1)} \to 0, \qquad \text{and} \qquad \lambda_+ \leq \alpha_k \leq \lambda_+ + M \epsilon_k, $$ 
and fail \eqref{E5} and \eqref{E6} for some  $\rho$ and $C$ which will be determined later. 
Note that by the second assumption above
$$ \ell < M \lambda_+^{p-1}<\infty. $$
We let $(v_k)_k$ be the sequence of functions defined in \eqref{E9} and assume that they converge to a function $v$ as in Lemma \ref{L5.5}, note that $\|v\|_{L^{\infty}(B_{\frac{1}{2}})} \leq \max(\frac 1{\lambda_+},\frac 1{\lambda_-})$. By Lemma \ref{L6}, $v$ solves \eqref{E13} and thus by Lemma \ref{LL2} there exist ${\bf  v} \in \mathbb{R}^{n-1}$, $s,t \in \mathbb{R}$ satisfying $\lambda_+^{p} s = \lambda_-^{p} t \geq -\ell $ such that for all $r \in (0,r_p)$
\begin{equation*}
\sup_{x \in B_{r}} \left|v(x)- \left(\textbf{v} \cdot x' + s x_n^+ - t x_n^- \right)\right|\leq  C (1+M)r^{\frac{3}{2}}.
\end{equation*}
Hence, we can fix $\rho=\rho(\lambda_{\pm},\gamma,L, M,p,n) < r_p$ such that $C (1+M)\rho^{\frac{1}{2}-\gamma} \le \frac12$, so
\begin{equation}
\label{EE4}
\sup_{x \in B_{\rho}} \left|v(x)- \left(\textbf{v} \cdot x' + s x_n^+ - t x_n^- \right)\right| \leq \frac{\rho^{1+\gamma}}{2L}.
\end{equation}
We now set
$$ \tilde{\alpha}_k:=\alpha_k(1+\epsilon_k s)+\delta_k \epsilon_k, \qquad \text{and} \qquad \textbf{e}_k:=\frac{\textbf{e}_n+\epsilon_k \textbf{v}}{\sqrt{1+\epsilon_k^2|\textbf{v}|^2}}, $$
where $\delta_k \to 0 $ is chosen so that $\tilde{\alpha}_k \geq \lambda_+$; note that the existence of such sequence is due to the condition $\lambda_+^{p} s \geq -\ell $ since
$$ \alpha_k (1+\epsilon_k s)=\left( \lambda_++ \frac{\ell}{\lambda_+^{p-1}} \epsilon_k+o(\epsilon_k) \right)\left(1+\epsilon_k s\right) \geq \lambda_+ + o(\epsilon_k). $$
We let $H_k:=H_{\tilde{\alpha}_k,\textbf{e}_k}$ and  note that 
$$ |\tilde \alpha_k-\alpha_k|+|\textbf{e}_k-\textbf{e}_n| \leq C \epsilon_k, $$
for a universal constant $C>0$; we also have used \eqref{EE4} to find out that $s$ is universally bounded.
By the contradiction assumption we have
\[ 
\begin{split}
\rho^{1+\gamma} < &\frac1{ \epsilon_k} \sup_{B_{\rho}} |u_k(x)-H_k(x)|  \\
\le &\max\left(\alpha_k \left\| v_k^+-\frac{H_k-H_{\alpha_k,\textbf{e}_n}}{ \epsilon_k\alpha_k}\right\|_{L^\infty(\Omega_{u_k}^+\cap B_\rho)} , 
\beta_k \left\| v_k^--\frac{H_k-H_{\alpha_k,\textbf{e}_n}}{ \epsilon_k\beta_k}\right\|_{L^\infty(\Omega_{u_k}^-\cap B_\rho)}
\right).
\end{split}
\]
To close the argument, we need to recall \eqref{EE4}, the convergence of $v_k$ to $v$ in the sense of Lemma \ref{L5.5} and the convergence of  (again in the sense of Lemma \ref{L5.5})
$$
\begin{cases}
\dfrac{H_k(x)-H_{\alpha_k,\textbf{e}_n}(x)}{  \alpha_k\epsilon_k}, & \qquad x_n>0, \\
\\
\dfrac{H_k(x)-H_{\alpha_k,\textbf{e}_n}(x)}{ \beta_k \epsilon_k}, & \qquad x_n<0,
\end{cases}
$$
 to the function
$$( \textbf{v} \cdot x')+sx_n^+-tx_n^-.$$
\end{proof}

\begin{proof}[Proof of Lemma \normalfont{\ref{L5}}]
Arguing by contradiction one assume for fixed $\gamma \in (0,1)$ the existence of a sequence of functions $u_k$ and numbers $\alpha_k, M_k \to \infty$ such that
$$ \epsilon_k=\|u_k- H_{\alpha_k,\textbf{e}_n}\|_{L^{\infty}(B_1)} \to 0, \qquad \text{and} \qquad \frac{\alpha_k-\lambda_+}{\epsilon_k} \geq M_k \to \infty, $$ 
and fail \eqref{E5} and \eqref{E6} for some  $\rho$ and $C$ which will be determined later. 
This implies that $\ell=\infty$ and that the limiting function $v$ obtained in Lemma \ref{L5.5} is a solution of \eqref{E12}. 
One then concludes the proof similar to the proof of Lemma \ref{L4} by using Lemma \ref{LL1}.
\end{proof}

\section{Regularity of the free boundary}
\label{S.4}
The last step in achieving the desired regularity result is to demonstrate that $|\nabla u^{\pm}|$ are $C^{0,\eta}$ for a suitable $\eta > 0$ up to the boundary, in the viscosity sense. Indeed, this shows that $u^{\pm}$ are solutions to the classical one-phase free boundary problem in its viscosity formulation and that the regularity will follow form \cite{MR4273843}. The arguments are similar to the ones in \cite[Section 8]{velichkov2019regularity} (see also \cite[Section 4]{MR4285137}). Therefore we only sketch the main steps and refer the reader to that paper for more details. 

Before stating the main results, we introduce some notation. 
For every $x_0 \in F(u)$ and every $0< r < \mathrm{dist}(x_0,\partial D)$, we consider the function 
$$ u_{x_0,r}(x):=\frac{u(x_0+rx)}{r}, $$
which is well-defined for $|x|< \frac{1}{r} \mathrm{dist}(x_0,\partial D)$ and vanishes at the origin. 
When $x_0=0$, we denote $u_{0,r}$ by $u_{r}$. 
Given a sequence $r_k>0$ such that $r_k \to 0$, we say that the sequence of functions $u_{x_0,r_k}$ is a \textit{blow-up sequence} of $u$ at $x_0$. 
If  a subsequence of $u_{x_0,r_k}$ convergs to $v$  on every ball $B_R \subset \mathbb{R}^n$, 
we say that $v$ is a blow-up limit of $u$ at $x_0$.

\begin{lemma}
\label{L10}
There exists $\bar \epsilon>0$ such that if the minimizer $u$ satisfies \eqref{T1-Eq1}, then
 at every point $x_0\in\Gamma_{\mathrm{TP}}\cap B_{r_0}$ for a universal radius $r_0>0$, there is a unique blow-up. 
Moreover, $u$ is Lipschitz in $B_{\frac{r_0}{2}}$ and there exists $\eta>0$  
and a constant $C_0(n,p,\Lambda_0,\Lambda_1)>0$ 
such that for every $x_0, y_0 \in \Gamma_{\mathrm{TP}} \cap B_{\frac{r_0}{2}}$, we have
\begin{equation}
\label{E15}
|\alpha(x_0)-\alpha(y_0)|\leq C_0|x_0-y_0|^{\eta}, \qquad \text{and} \qquad |\textbf{e}(x_0)-\textbf{e}(y_0)|\leq C_0|x_0-y_0|^{\eta},
\end{equation} 
for any $\eta\in(0, \frac13)$, where $H_{\alpha(x_0), \textbf{e}(x_0)}$ and $H_{\alpha(y_0), \textbf{e}(y_0)}$ are the blow-ups at $x_0$ and $y_0$, respectively.  
\end{lemma}
\begin{proof}
Let $L_0=\Lambda_0$ and $L_1=2\Lambda_1$ in Theorem \ref{T2} and find the universal constants $\epsilon_0>0$, $\rho_0>0$ and $C>0$.
Choose $\bar \epsilon <\min \{(1-\rho^\gamma) \frac{\Lambda_1}{2C}, \frac{\epsilon}{2}\}$ and $r_0 < \frac{\bar \epsilon}{\Lambda_1}$, then if  the minimizer $u$ satisfies \eqref{T1-Eq1} for some $\textbf{e}\in \mathbb{S}^{n-1}$ and $\lambda_+\le  \alpha \le \Lambda_1$, then 
\[
\left\|u_{x_0,\frac12}-H_{\alpha, \textbf{e}} \right\|_{L^{\infty}(B_1)} \leq \bar\epsilon + |H_{\alpha, \textbf{e}}(x_0)| \le  2\bar\epsilon,
\]
for any $x_0\in \Gamma_{\mathrm{TP}}\cap B_{r_0}$.
Now we can thus repeatedly apply Theorem \ref{T2} to obtain the sequences $u_{x_0,\frac{\rho^k}2}(x)= \frac 2{\rho^k}u(x_0 + \frac{\rho^k}2 x),\ \max (L_0, \lambda_+) \le \alpha_k\le L_1 $ and $\textbf{e}_k\in \mathbb{S}^{n-1}$ that
\[
\left\|u_{x_0,\frac{\rho^k}2}  - H_{\alpha_k, \textbf{e}_k} \right\|_{L^{\infty}(B_1)} \le 2\bar\epsilon \rho^{k\gamma}, \qquad |\textbf{e}_{k+1}-\textbf{e}_k|+ |\alpha_{k+1}-\alpha_k| \le 2C\bar\epsilon \rho^{k\gamma}.
\]
This implies that $\alpha_k$ and $\textbf{e}_k$ converge to some $\alpha=\alpha(x_0)$ and $\textbf{e}=\textbf{e}(x_0)$, respectively.

Now let $r\le \frac{1}{2}$ be arbitrary and choose $k\in \bN$ such that $\rho^{k+1} \le 2r \le \rho^k$, then
\[\begin{split}
& \left\|u_{x_0,r}  - H_{\alpha(x_0), \textbf{e}(x_0)}\right\|_{L^{\infty}(B_1)} \le \frac1\rho \left\|u_{x_0,\frac{\rho^k}2}  - H_{\alpha(x_0), \textbf{e}(x_0)}\right\|_{L^{\infty}(B_1)} \\
&\quad \le  \frac1\rho \left( \left\|u_{x_0,\frac{\rho^k}2}  - H_{\alpha_k, \textbf{e}_k} \right\|_{L^{\infty}(B_1)} + \left\|H_{\alpha_k, \textbf{e}_k}  - H_{\alpha(x_0), \textbf{e}(x_0)}\right\|_{L^{\infty}(B_1)} \right) \le C\bar\epsilon \rho^{k\gamma}.
\end{split}\]
Therefore, there is $\tilde C= \tilde C(n,p,\Lambda_0,\Lambda_1)$ such that for every $r\le \frac{1}{2}$ and $x_0\in B_{r_0}$,
\begin{equation}
\label{E16}
\left\| u_{x_0,r}-H_{\alpha(x_0), \textbf{e}(x_0)} \right\|_{L^{\infty}(B_1)} \leq \tilde C r^{\gamma},
\end{equation}
where $\gamma \in (0,\frac{1}{2})$. 

According to \eqref{E16},
\[
\norm{u}_{L^\infty(B_r(x_0))} \le (L_1 + \tilde C) r,\qquad r\le \frac12,
\]
for every $x_0 \in \Gamma_{\mathrm{TP}} \cap B_{r_0}$. 
From this and the Lipschitz regularity around the one-phase points, Proposition \ref{pro:lip-one-phase}, we conclude that $u$ is Lipschitz in $B_{\frac{r_0}{2}}$; see \cite[Theorem 2.3]{MR752578} or \cite[Theorem 2.1]{MR1976085}. 

Next, for $x_0,y_0 \in \Gamma_{\mathrm{TP}} \cap B_{\frac{r_0}{2}}$ set $r:=|x_0-y_0|^{1-\eta}$ and $\eta:=\frac{\gamma}{1+\gamma}$, and recall that $u$ is Lipschitz (with a constant $\tilde L$) to get
$$
\begin{aligned}
& \left\| H_{\alpha(x_0), \textbf{e}(x_0)} - H_{\alpha(y_0), \textbf{e}(y_0)} \right\|_{L^{\infty}(B_1)} \\
&  \leq \left \|u_{x_0,r}-H_{\alpha(x_0), \textbf{e}(x_0)} \right\|_{L^{\infty}(B_1)} + \left\|u_{x_0,r}-u_{y_0,r} \right\|_{L^{\infty}(B_1)}+ \left\|u_{y_0,r}-H_{\alpha(y_0), \textbf{e}(y_0)} \right\|_{L^{\infty}(B_1)} \\
&  \leq \left(C_0r^{\gamma}+\frac{\tilde L}{r}|x_0-y_0|+C_0r^{\gamma} \right)=(\tilde L+2C_0)|x_0-y_0|^{\eta}.
\end{aligned}
$$
The conclusion now follows easily from this inequality; see e.g. \cite[Lemma 8.8]{velichkov2019regularity} for the details.
\end{proof}

\begin{lemma}
\label{L11}
Under the same assumptions of Lemma \ref{L10}, there are $C^{0,\eta}$ continuous functions $\alpha: \partial \Omega^+_u \to \mathbb{R}$ and $\beta: \partial \Omega^-_u \to \mathbb{R}$ such that $\alpha \geq \lambda_{+}$, $\beta \geq \lambda_{-}$ and $u^{\pm}$ are viscosity solutions of the one-phase problems
$$ \Delta_p u^+=0 \qquad \mathrm{in} \quad \Omega^+_u, \qquad |\nabla u^+|=\alpha \qquad \mathrm{on} \quad \partial \Omega^+_u, $$
and
 $$ \Delta_p u^-=0 \qquad \mathrm{in} \quad \Omega^-_u, \qquad |\nabla u^-|=\beta \qquad \mathrm{on} \quad \partial \Omega^-_u. $$
\end{lemma}
\begin{proof}
We will sketch the argument for $u^+$. The proof of the case $u^-$ is similar. Clearly $\Delta_p u^+ =0$ in $\Omega^+_u$. By \eqref{E16} we have that, if $x_0 \in \Gamma_{\mathrm{TP}}\cap D'$, then
\begin{equation}
\label{E17}
\left|u^+(x)-\alpha(x_0)\left((x-x_0) \cdot \textbf{e}(x_0)\right)^+\right| \leq C_0 |x-x_0|^{1+\gamma},
\end{equation}
for every  $x \in B_{r_0}(x_0) \cap \Omega^+_u$
where $r_0$ and $C_0$ depends only on $D'$. In particular, $u^+$ is differentiable on $\Omega^+_u$ up to $x_0$ (in the classical sense) and $|\nabla u^+(x_0)|=\alpha(x_0)$. On the other hand if $x_0\in \Gamma^+_{\mathrm{OP}}$, then $|\nabla u^+(x_0)|=\lambda_{+}$ is constant, in the viscosity sense.

To close the argument, we only need to prove that $\alpha \in C^{0,\eta}(\partial \Omega^+_u)$. Since $\alpha$ is $\eta$-H\"older continuous on $\Gamma_{\mathrm{TP}}$ by Lemma \ref{L10}, and constant on $\Gamma^+_{\mathrm{OP}}$ (in the viscosity sense), we just need to show that if $x_0 \in \Gamma_{\mathrm{TP}}$ is such that there is a sequence $x_k \in \Gamma^+_{\mathrm{OP}}$ converging to $x_0$, then $\alpha(x_0)=\lambda_{+}$. To this end, let $y_k \in \Gamma_{\mathrm{TP}}$ be such that 
$$ \mathrm{dist}(x_k,\Gamma_{\mathrm{TP}})=|x_k-y_k|, $$
and denote
$$ r_k=|x_k-y_k|, \qquad \text{and} \qquad u_k(x)=\frac{1}{r_k}u^+(x_k+r_kx), $$
and note that $u_k$ is a viscosity solution of the free boundary problem
$$ \Delta_p u_k = 0 \qquad \text{in} \quad \Omega_{u_k}^+ \cap B_1, \qquad |\nabla u_k|=\lambda_{+} \qquad \text{on} \quad \partial\{u_k>0\} \cap B_1. $$
Since $u_k$ are uniformly Lipschitz in $B_{\frac{1}{2}}$ (Proposition \ref{pro:lip-one-phase}) they converge to a function $u_{\infty}$ which is also a viscosity solution of the same problem (see e.g. \cite{MR4273843}). 
On the other hand, by \eqref{E17} for two-phase point $y_k \in \Gamma_{\mathrm{TP}}$ and letting $z_k:=\frac{x_k-y_k}{r_k}$, we have that 
\begin{equation*}
\left|u_k(x)-\alpha(y_k)\left((x-z_k) \cdot \textbf{e}(y_k)\right)^+\right| \leq C_0r_k^\gamma |x-z_k|^{1+\gamma}.
\end{equation*}
Suppose $z_k\to z_0$  and passing to the limit
$$ u_{\infty}(x)=\alpha(x_0) \left((x-z_0) \cdot \textbf{e}(x_0)\right)^+, \qquad \text{in} \quad B_{\frac{1}{2}},$$
which gives that $\alpha(x_0)=|\nabla u_\infty(0)|= \lambda_{+}$.
\end{proof}

\begin{proof}[Proof of Theorem \normalfont{\ref{T1}}]
Let $0 \in \Gamma_{\mathrm{TP}}=\partial\Omega^+_u \cap \partial\Omega^-_u$ and let $\overline{\epsilon}$ be the constant satisfies in \cite[Theorem 1.1]{MR4273843} and Lemma \ref{L10}. 
By virtue of Lemma \ref{L11}, we can apply \cite[Theorem 1.1]{MR4273843} to conclude that locally at $0 \in \Gamma_{\mathrm{TP}}$ the free boundaries $\partial \Omega^{\pm}_u$ are $C^{1,\eta}$ graphs. 
\end{proof}


\section{Lipschitz regularity of solutions}

In this section, we are going to prove Theorem \ref{P1}. We will follow the idea in \cite{MR3938321}.
\begin{proposition}\label{prop:dichotomy}
Let $u:D \to \mathbb{R}$ be a minimizer of $J_{\mathrm{TP}}$ that   $0 \in F(u) \cap B_1 \subset D$. 
Then there exists constants  $L$ and $\delta$ such that one of the following alternative holds:
\begin{enumerate}
\item $u$ is Lipschitz in $B_\delta$ and 
\[
|\nabla u| \le C\max(\norm{u}_{L^\infty(B_1)}, L), \qquad \text{in} \quad B_\delta,
\]
for some universal constant $C$.

\item  $\frac 1\delta \norm{u}_{L^\infty(B_\delta)} \le \frac12 \max(\norm{u}_{L^\infty(B_1)}, L)$.
\end{enumerate}
\end{proposition}
\begin{proof}
Let $\delta$ be fixed, to be specified later. 
Assume by contradiction that there exist a sequence of $L_j\to\infty$ and a sequence of solutions $u_j$ such that does not satisfy either $(1)$ nor $(2)$.
Let $C_j:=\max(\norm{u_j}_{L^\infty(B_1)}, L_j)$  and define
\[
\tilde u_j:= \frac{u_j}{C_j},
\]
which satisfy
\[
\norm{\tilde u_j}_{L^\infty(B_1)} \le 1, \qquad \text{and}\qquad \norm{\tilde u_j}_{L^\infty(B_\delta)} \ge \frac\delta2.
\]
By the same proof of Proposition \ref{LConverg}, we get that $\tilde u_j$ is a minimizer of the scaled functional \eqref{LEq17} for $\sigma_j = \frac{1}{C_j} \to 0$. 
Thus up to a subsequence, $\tilde u_j$ converges uniformly to a $p$-harmonic function $u_0$. 
Hence by $C^{1,\alpha}$ regularity for $p$-harmonic functions we get that
\begin{equation}\label{Lip:Eq1}
\sup_{B_r}\left|u_0(x)-\nabla u_0(0)\cdot x\right| \le \tilde Cr^{1+\alpha}, \qquad \text{ for all }r\le 1,
\end{equation}
where the constant $C$ is universal and also $|\nabla u_0(0)| \le \tilde C$.
Now we distinguish two cases:

\noindent{\bf Case I:} $|\nabla u_0(0)| \le \frac14$. \\
In this case, from \eqref{Lip:Eq1} we deduce that
\[
\frac1\delta \norm{u_0}_{L^\infty(B_\delta)} \le \frac1 4 + \tilde C\delta^{\alpha} \le \frac13,
\]
if we choose $\delta$ small enough. Thus all $u_j$ for sufficiently   large $j$ will satisfy $(2)$, which is a contradiction.

\noindent{\bf Case II:} $|\nabla u_0(0)| \ge \frac14$. \\
In this case we will use our flatness result in Theorem  \ref{T1}. 
Put $\tilde r=\frac{2\delta}{r_0}$ in \eqref{Lip:Eq1} where $r_0$ is the radius obtained in Theorem \ref{T1} (we have also assumed $2\delta \le r_0$)
\[
\sup_{B_{1}}\left|\tilde u_{j,\tilde r}(x)-\nabla u_0(0)\cdot x\right| \le \frac{\tilde C}{r_0^{\alpha}}\delta^{\alpha}.
\]
Now, let $\textbf{e} = \frac{\nabla u_0(0)}{|\nabla u_0(0)|}$, $\alpha = |\nabla u(0)|$ and $\beta_j= \frac1{C_j} \left(\lambda_-^p -\lambda_+ ^p + \alpha^p C_j^p \right)^{\frac{1}{p}}$, then
\[
\left\| u_{j,\tilde r} - H_{\alpha, \textbf{e}} \right\|_{L^\infty(B_1)} \le  \left\|u_{j,\tilde r} - \nabla u_0(0)\cdot x \right\|_{L^\infty(B_1)} + |\alpha-\beta_j| \le \frac{2\tilde C}{r_0^{\alpha}}\delta^{\alpha},
\]
for sufficiently large $j$.
Applying Theorem \normalfont{\ref{T1}} for some $\Lambda_0 \le \frac14 $ and $  \Lambda_1 \ge \tilde C$ and notice that $u_{j,\tilde r}$ is a minimizer of $J_{\mathrm{TP}}$ for coefficients $\frac1{C_j}\lambda_\pm$.
Note that the critical flatness in  Theorem \ref{T1} or Lemma \ref{L10} depends on $\Lambda_0$ and $\Lambda_1$  rather than coefficients $\lambda_\pm$. 
Then, we can find $\delta$ universally small such that  $u_{j,\tilde r}$ satisfy in Lemma \ref{L10}. In particular, $u_{j,\tilde r}$ is Lipschitz in $B_{\frac{r_0}{2}}$ with a universal constant. It proves that $u_j$ is Lipschitz in $B_\delta$.
\end{proof}

\begin{proof}[Proof of Theorem \normalfont{\ref{P1}}]
Let $\delta$, $C$ and $L$ be the universal constants in Proposition \ref{prop:dichotomy}.
Assume $0\in F(u)$ and let $\tilde L:= \max(\norm{u}_{L^\infty(B_1)}, L)$. We first show 
\begin{equation}\label{Lip:Eq2}
\norm{u}_{L^\infty(B_{\delta^k})} \le C\tilde L\delta^k, \qquad \forall k\ge 0.
\end{equation}
By Proposition  \ref{prop:dichotomy} either $(1)$ or $(2)$ holds. 
In the first case, $u$ is Lipschitz in $B_\delta$ and 
\[
|\nabla u|\le C\tilde L, \qquad \text{in} \quad B_\delta.
\]
Thus \eqref{Lip:Eq2} holds for all $k\ge 1$.

If $(2)$ holds, then 
\[
\norm{u}_{L^\infty(B_{\delta})} \le \frac{\tilde L}2\delta.
\]
We now rescale and iterate. Define 
\[
u_k(x):=\frac{u(\delta^k x)}{\delta^k},
\]
which is also a minimizer of  $J_{\mathrm{TP}}$ and we can apply Proposition \ref{prop:dichotomy}.
If $k_0$ is the smallest $k$ for which $u_k$ satisfies $(1)$, then for $0\le k < k_0$ the item $(2)$ holds and so
\[
\norm{u}_{L^\infty(B_{\delta^k})} \le \tilde L\delta^k, \qquad \text{for} \quad 0\le k \le k_0.
\]
Moreover, $u_{k_0}$ is Lipschitz in $B_\delta$, with
\[
|\nabla u_{k_0}|\le C \max(\norm{u_{k_0}}_{L^\infty(B_1)},L) \le C\max(\tilde L, L) = C\tilde L, \qquad \text{in} \quad B_\delta.
\]
Hence, \eqref{Lip:Eq2} holds for all $k\ge k_0$.
If  $u_k$ satisfy the alternative $(2)$ for all $k$,  the estimate  \eqref{Lip:Eq2}  will be obtained easily. 

Now for an arbitrary $r$ choose $k$ such that $\delta^{k+1} \le r \le \delta^k$, then by \eqref{Lip:Eq2} we get
\[
\norm{u}_{L^\infty(B_{r})} \le \norm{u}_{L^\infty(B_{\delta^k})} \le C\tilde L \delta^k \le \frac{C\tilde L}{\delta} r.
\]
This is enough to obtain the Lipschitz continuity locally in $D$.
\end{proof}


\appendix

\section{Proof of Proposition \ref{LConverg}  }\label{appendix}

\begin{proof}[Proof of Proposition \ref{LConverg}]
By definition of $v_j$ and an easy computation, we get
$$ \nabla v_j(x)=\frac{r_j}{S_j} \nabla u_j(x_j+r_jx)=\sigma_j  \nabla u_j(x_j+r_jx). $$
In order to show that $v_j$ is a minimizer of $\hat{J}_{\mathrm{TP}}$ in $B_R$, consider $w$ that
\[
\hat{J}_{\mathrm{TP}}(w, B_R) < \hat{J}_{\mathrm{TP}}(v_j, B_R), \qquad  \text{and} \qquad w=v \quad \text{on }\partial B_R.
\] 
Then $\hat w(x) = w(\frac{x-x_j}{r_j})$ will satisfy $\hat w = u_j$ on $\partial B_{r_j}(x_j)$ and
by a simple calculation, we get that 
\[
 J_{\mathrm{TP}}(\hat w, B_{r_j})=\frac{r_j^{n}}{\sigma_j^p}\hat{J}_{\mathrm{TP}}(w, B_R) < \frac{r_j^{n}}{\sigma_j^p}\hat{J}_{\mathrm{TP}}(v_j, B_R) =  \hat{J}_{\mathrm{TP}}(u_j, B_{r_j}).
\]
This is a contradiction with the minimality of $u_j$.

\medskip

Moreover, using $|v_j|\leq M$ in $B_{\frac{4R}{3}}$ and Caccioppoli's inequality, we conclude that
$$ \int_{B_{R}} |\nabla v^{\pm}_j|^p \, dx \leq 4^p C(n) \int_{B_{\frac{4R}{3}}} (v_j^{\pm})^p\, dx \leq (4M)^{p} C(n), $$
for some $C(n)>0$, indicating that $\|v_j\|_{W^{1,p}(B_R)}$ are uniformly bounded. Hence, from Proposition \ref{P0}, i.e. the BMO estimate for the gradient, we obtain that for any $q>1$ and $0<R<\frac{1}{r_j}$ there exists a constant $C=C(R,q)>0$ independent of $j$ such that
$$ \max \left\{ \|v_j\|_{C^{\alpha}(B_{R})}, \|\nabla v_j\|_{L^q(B_{R})} \right\} \leq C, $$
for some $\alpha \in (0,1)$ (if $q>n$, one can take $\alpha=1-\frac{n}{q}$ by the Morrey's inequality). Therefore, by a standard compactness argument, we have that, up to a subsequence,
$v_j$ converges to some function $v_0$ as $j \to + \infty$ in $C^{\alpha}(B_R)$
and weakly in $W^{1,q}(B_R)$ for any $q>1$, and for any fixed $R$.
This completes the proof of \textup{(i)}.

\medskip

For obtaining \textup{(ii)}, firstly, we prove that $ \Delta_p v_0= 0$ in the positivity set of $v_0$. Let $E \Subset \{v_0>0\}$. Then, there exists $c>0$ such that $v_0 \geq 2c$ in $E$. By the uniform convergence of $v_j$ to $v_0$, we will have $v_j > c$ in $E$ for large $j$. This implies that $v_0$ is $p$-harmonic in $E$. Since $E$ was arbitrary, we are done. Now, take $0 \leq \varphi \in C_c^1(\mathbb{R}^n)$ and $s>0$. By using $(v_0-s)^+ \varphi$ as a test function in the weak formulation of $ \Delta_p v_0= 0$ in the set $\{v_0>0\}$, we have
$$ \int_{\{v_0>s\}} |\nabla v_0|^p \varphi \, dx = - \int_{\{v_0>s\}} |\nabla v_0|^{p-2} \left( \nabla v_0 \cdot \nabla \varphi \right) v_0 \, dx + s \int_{\{v_0>s\}} |\nabla v_0|^{p-2} \nabla v_0 \cdot \nabla \varphi \, dx. $$
Letting $s \to 0$ gives that
\[
\int_{\{v_0>0\}} |\nabla v_0|^p \varphi \, dx = - \int_{\{v_0>0\}} |\nabla v_0|^{p-2} \left( \nabla v_0 \cdot \nabla \varphi \right) v_0 \, dx.
\]
Similar argument holds for using test function $(v_0+s)^- \varphi$ and finally we get
\begin{equation}
\label{MAup-c1}
\int_{\mathbb{R}^n} |\nabla v_0|^p \varphi \, dx = - \int_{\mathbb{R}^n} |\nabla v_0|^{p-2} \left( \nabla v_0 \cdot \nabla \varphi \right) v_0 \, dx.
\end{equation}
On the other hand since $v_j \varphi \Delta_p v_j \geq 0$ (Theorem \ref{EoM}), we have
\begin{equation}
\label{MAup-c2}
\int_{\mathbb{R}^n} |\nabla v_j|^p \varphi \, dx \leq - \int_{\mathbb{R}^n} |\nabla v_j|^{p-2} \left(\nabla v_j \cdot \nabla \varphi \right) v_j\, dx.
\end{equation}
Using the uniform convergence of $v_j$ to $v_0$ and the weak convergence of $|\nabla v_j|^{p-2} \nabla v_j \rightharpoonup |\nabla v_0|^{p-2} \nabla v_0$ in $L_{\mathrm{loc}}^{\frac{p}{p-1}}(\mathbb{R}^n)$ (see \cite{MR1246185}), we infer from \eqref{MAup-c1} and \eqref{MAup-c2} that
\begin{equation}
\label{MAup-c3}
\limsup_{j \to + \infty} \int_{\mathbb{R}^n} |\nabla v_j|^p \varphi \, dx \leq \int_{\mathbb{R}^n} |\nabla v_0|^p \varphi \, dx.
\end{equation}
Since also $ \nabla v_j \rightharpoonup \nabla v_0$ weakly in $L_{\mathrm{loc}}^{p}(\mathbb{R}^n)$, we have
\begin{equation}
\label{MAup-c4}
\int_{\mathbb{R}^n} |\nabla v_0|^p \varphi \, dx \leq \liminf_{j \to + \infty} \int_{\mathbb{R}^n} |\nabla v_j|^p \varphi \, dx.
\end{equation}
It follows from \eqref{MAup-c3}, \eqref{MAup-c4}, and a simple compactness argument that 
\begin{equation}
\label{STRONG}
|\nabla v_j|^p \to |\nabla v_0|^p, \qquad \text{strongly in \,\,\, $L_{\mathrm{loc}}^1(\mathbb{R}^n)$},
\end{equation}
and we get \textup{(ii)}. 

\medskip

Finally, we prove the claim \textup{(iii)}. For this, notice that for any $\psi \in C^{\infty}_c(B_R)$
\begin{equation}
\label{LEq20}
\begin{aligned}
& \int_{B_R} |\nabla v_j|^p  + \sigma_j^p  (p-1) \lambda_+^p \chi_{\{v_j>0\}} + \sigma_j^p  (p-1) \lambda_-^p \chi_{\{v_j<0\}} \, dx \\
& \ \  \leq \int_{B_R}|\nabla (v_j+\psi )|^p + \sigma_j^p  (p-1) \lambda_+^p \chi_{\{v_j+\psi >0\}} + \sigma_j^p  (p-1) \lambda_-^p \chi_{\{v_j+\psi <0\}} \, dx,
\end{aligned}
\end{equation}
because $v_j$ is a minimizer for $\hat{J}_{\mathrm{TP}}$ defined in \eqref{LEq17}. 
Recall the strong convergence \eqref{STRONG} along with the following standard inequality  
$$ |\nabla (v_j+\psi )|^p \leq 2^{p-1} \left( |\nabla v_j|^p + |\nabla \psi|^p \right) ,  $$
we get 
$$ \int_{B_R} |\nabla (v_j+\psi )|^p \, dx\to \int_{B_R} |\nabla (v_0+\psi )|^p\, dx, $$
as $j \to + \infty$. 
Thus passing \eqref{LEq20} to limit, we have 
$$
\begin{aligned}
 \int_{B_R}& |\nabla v_0|^p + \sigma^p  (p-1) \lambda_+^p \chi_{\{v_0>0\}} + \sigma^p  (p-1) \lambda_-^p \chi_{\{v_0<0\}} \, dx \\
& \leq  \int_{B_R} |\nabla(v_0+\psi )|^p + \sigma^p  (p-1) \lambda_+^p \chi_{\{v_0+\psi>0\}} + \sigma^p  (p-1) \lambda_-^p \chi_{\{v_0+\psi<0\}} \, dx, 
 \end{aligned}
 $$
for any $\psi  \in C^{\infty}_c(B_R)$. This implies \textup{(iii)} and the proof of proposition finishes.
\end{proof}

\section{Proof of the regularity for two-membrane problem} \label{appendix:auxilarly}

\begin{proof}[Proof of  Lemma \ref{LL2}]
For the given solution $v$, define $w$
$$ w_{\pm}(x)=v_{\pm}(x', (p-1)^{\frac{1}{2}}x_n)+\frac{\ell(p-1)^{\frac{1}{2}}}{\lambda_{\pm}^p}x_n, \qquad x \in B^{\pm}_{2r_p}. $$
It is straightforward to check that $w_{\pm}$ is a viscosity solution of
\begin{equation*}
\begin{cases}
\mathcal{L}_p(w_{\pm})= 0, \qquad & \mathrm{in} \quad B^{\pm}_{2r_p}, \\
\partial_n w_{\pm} \geq 0,  \qquad & \mathrm{in} \quad B_{2r_p} \cap \{x_n=0\}, \\
\partial_n w_{\pm} = 0,  \qquad & \mathrm{in} \quad \mathcal{J}=\{w_+<w_-\}\cap\{x_n=0\}, \\
\lambda_{+}^{p} \partial_n w_+ = \lambda_{-}^{p} \partial_n w_-, \qquad & \mathrm{in} \quad \mathcal{C}=\{w_+=w_-\}\cap \{x_n=0\}, \\
w_+ \leq w_-, \quad & \mathrm{in} \quad B_{2r_p} \cap \{x_n=0\}. \\
\end{cases}
\end{equation*}
Furthermore one can easily check that
\begin{equation} \label{E13-Com}
w_{\pm}(x',x_n)= \tilde w(x',\mp x_n) \mp \frac{1}{\lambda_{\pm}^p} w_S(x',\mp x_n), 
\end{equation}
where $\tilde w$ solves the following Neumann problem
$$
\begin{cases}
\Delta \tilde w  = 0, \qquad & \mathrm{on} \quad B^{-}_{2r_p}, \\
\partial_n \tilde w = 0,  \qquad & \mathrm{on} \quad B^{-}_{2r_p} \cap \{x_n=0\}, 
\end{cases}
$$
and $w_S$ is a solution to the thin obstacle (the Signorini) problem
$$
\begin{cases}
\Delta w_{S} = 0, \qquad & \mathrm{on} \quad B^{-}_{2r_p}, \\
w_{S} \geq 0,  \qquad & \mathrm{on} \quad B^{-}_{2r_p} \cap \{x_n=0\},\\
\partial_n w_{S} \geq 0,  \qquad & \mathrm{on} \quad B^{-}_{2r_p} \cap \{x_n=0\},\\
w_S \partial_n w_{S} = 0,  \qquad & \mathrm{on} \quad B^{-}_{2r_p} \cap \{x_n=0\}.
\end{cases}
$$
The boundary data of $\tilde w$ and $w_S$ on $\partial B_{2r_p}\cap \{x_n<0\}$ will be obtained uniquely from  \eqref{E13-Com}.
Clearly $\tilde w \in C^{\infty}(\overline{B^-_{r_p}})$ with
$$ \|\tilde w\|_{C^k(B_{r_p}^-)} \leq C_k \| \tilde w \|_{L^{\infty}(B_{2r_p}^-)}. $$
On the other hand, by \cite{MR2120184}, $w_S \in C^{1,\frac{1}{2}}(\overline{B^-_{r_p}})$ with
$$ \|w_S\|_{C^{1,\frac{1}{2}}(B_{r_p}^-)} \leq C \| w_{S} \|_{L^{\infty}(B_{2r_p}^-)}. $$
From the last two estimates and the definition of $w$, it is easy to deduce the conclusion of the lemma for (note that the positivity of $w_S$ along with its regularity necessitates that $\nabla'  w_S(0) =0$)
\[
{\bf  v} := \nabla' \tilde w(0)   \qquad \text{and} \qquad 
s^\pm :=  \frac{(p-1)^{-\frac{1}{2}}}{\lambda_{\pm}^p} \partial_nw_S(0)  - \frac{\ell}{\lambda_{\pm}^p}.
\]
\end{proof}

\section{Proof of non-degeneracy}\label{appendix:nondegeneracy}

\begin{proof}[Proof of Proposition \ref{P1.5}]
We will prove that for any $k \in (0,1)$, there exists a constant $c_k > 0$ such that for any local minimizer of $J_{\mathrm{TP}}$ and for any small ball $B_r(x_0) \subset D$
$$ \mathrm{if} \quad \frac{1}{r} \left( \avint_{B_r(x_0)} (u^{\pm})^p \, dx \right)^{\frac{1}{p}} < c_k \quad \text{then $u^\pm \equiv 0$ in $B_{kr}(x_0)$}.$$
By symmetry of the problem, we prove only the case $u^+$. Also, by the scale invariance, we can take $r=1$ and $x_0=0$ for simplicity. Now, let define
\begin{equation*}
\varepsilon:=\frac{1}{\sqrt{k}} \sup_{B_{\sqrt{k}}} u^+.    
\end{equation*}
Since $u^+$ is $p$-subharmonic, then by \cite[Theorem 3.9]{MR1461542}
$$ \varepsilon \leq \frac{1}{\sqrt{k}} \frac{C(n,p)}{(1-\sqrt{k})^{\frac{n}{p}}} \left( \avint_{B_1} (u^{+})^p \, dx \right)^{\frac{1}{p}}.$$
Also, let
$$
v(x):=\begin{cases}
C_1 \varepsilon \left( e^{-\mu |x|^2}-e^{-\mu k^2} \right), \quad & \mathrm{in} \quad B_{\sqrt{k}}\setminus B_k, \\
0, \quad & \mathrm{in} \quad B_k,
\end{cases}
$$
where $\mu>0$ and $C_1$ are such that  
\begin{equation}
\label{Non-D2}
v_{\big|_{\partial B_{\sqrt{k}}}}:=\sqrt{k} \varepsilon = \sup_{B_{\sqrt{k}}} u^+ \geq u_{\big|_{\partial B_{\sqrt{k}}}}.
\end{equation}
By direct computation, it is straightforward to check that
\begin{equation*}
\nabla v (x) = - 2C_1 \varepsilon  \mu x e^{-\mu |x|^2} \qquad \mathrm{in} \quad B_{\sqrt{k}}\setminus B_k,
\end{equation*}
and
$$
\Delta_pv(x)=C_1 \varepsilon (p-1) (2\mu)^2 |\nabla v|^{p-2} e^{-\mu |x|^2} \left( |x|^2 - \frac{n+p-2}{2\mu (p-1)} \right).
$$
Thus $v$ is nonnegative $p$-superharmonic in $B_{\sqrt{k}}\setminus B_k$,
if $\mu$ is sufficiently small, say,
\begin{equation}
\label{Non-D5}
\mu < \frac{n+p-2}{2k(p-1)}.
\end{equation}
On the other hand, since 
\begin{equation*}
w:=\min(u, v) = u \qquad \text{on} \quad \partial B_{\sqrt{k}},
\end{equation*}
thanks to \eqref{Non-D2}, by invoking the minimality of $u$, we get
\begin{equation}
\label{Non-D6}
J_{\mathrm{TP}}(u, B_{\sqrt{k}}) \leq J_{\mathrm{TP}}(w, B_{\sqrt{k}}).
\end{equation}
Now, since we have
$$
\begin{aligned}
J_{\mathrm{TP}}(w, B_{\sqrt{k}})=&\int_{B_k} |\nabla w|^p
 +(p-1)\lambda_+^p \chi_{\{w>0\}}+(p-1)\lambda_-^p \chi_{\{w<0\}} \, dx \\
& +  \int_{B_{\sqrt{k}} \setminus B_k} |\nabla w|^p 
 +(p-1)\lambda_+^p \chi_{\{w>0\}}+(p-1)\lambda_-^p \chi_{\{w<0\}} \, dx \\
=&  \int_{B_k \cap \{u \le 0 \}} |\nabla u|^p 
 +(p-1)\lambda_+^p \chi_{\{u>0\}}+(p-1)\lambda_-^p \chi_{\{u<0\}} \, dx \\
& + \int_{B_{\sqrt{k}} \setminus B_k} |\nabla w|^p
 +(p-1)\lambda_+^p \chi_{\{w>0\}}+(p-1)\lambda_-^p \chi_{\{w<0\}} \, dx.
\end{aligned}
$$
Therefore, from \eqref{Non-D6}, $\{w>0\}\subseteq\{u>0\}$ and $\{w<0\}=\{u<0\}$, we have that
$$
\begin{aligned}
& \int_{B_k \cap \{u > 0 \}}   |\nabla u|^p+ (p-1)\lambda_+^p \chi_{\{u>0\}}+(p-1)\lambda_-^p \chi_{\{u<0\}} \, dx \\
& \qquad \qquad \leq \int_{B_{\sqrt{k}} \setminus B_k} |\nabla w|^p 
+(p-1)\lambda_+^p \chi_{\{w>0\}}+(p-1)\lambda_-^p \chi_{\{w<0\}} \, dx \\
& \qquad \qquad  - \int_{B_{\sqrt{k}} \setminus B_k} |\nabla u|^p+(p-1)\lambda_+^p \chi_{\{u>0\}}+(p-1)\lambda_-^p \chi_{\{u<0\}} \, dx \\
 & \qquad \qquad \leq \int_{B_{\sqrt{k}} \setminus B_k} |\nabla w|^p-|\nabla u|^p \, dx
=\int_{(B_{\sqrt{k}} \setminus B_k)\cap \{u>v\}} |\nabla v|^p-|\nabla u|^p \, dx \\
& \qquad \qquad \leq -p \int_{B_{\sqrt{k}} \setminus B_k} |\nabla v|^{p-2} \nabla v \cdot \nabla \max(u-v,0) \, dx \\
& \qquad \qquad = -p \int_{B_{\sqrt{k}} \setminus B_k} - \Delta_p v \, \max(u-v,0) + \mathrm{div} \left( |\nabla v|^{p-2} \nabla v \, \max(u-v,0) \right) \, dx \\
 & \qquad \qquad \le -p \int_{B_{\sqrt{k}} \setminus B_k} \mathrm{div} \left( |\nabla v|^{p-2} \nabla v \, \max(u-v,0) \right) \, dx \\
& \qquad \qquad = p \int_{\partial B_{k}} |\nabla v|^{p-2} (\nabla v \cdot \nu) u^+,
\end{aligned}
$$
where to get  the last inequality we have used the fact that $v$ is a $p$-superharmonic in $B_{\sqrt{k}} \setminus B_k$. 
Moreover, by \eqref{Non-D5}, we have that $|\nabla v|=2C_1 \varepsilon \mu k e^{-\mu k^2} \leq C \varepsilon$ on $\partial B_k$, for some $C>0$. Thus
\begin{equation}
\label{Non-D7}
\int_{B_k \cap \{u > 0 \}} |\nabla u|^p+(p-1)\lambda_+^p \chi_{\{u>0\}} \, dx \leq p(C\varepsilon)^{p-1} \int_{\partial B_k} u^+.
\end{equation}
On the other hand, from trace estimate, Young's inequality, we get
\begin{equation}
\label{Non-D8}
\begin{aligned}
\int_{\partial B_k} u^+ &\leq C(n,k) \left( \int_{B_k} u^+ \, dx + \int_{B_k} |\nabla u^+| \, dx \right) \\
& \leq C(n,k) \left( \sup_{B_k} u^+ \int_{B_k} \chi_{\{u>0\}} \, dx + \int_{B_k} \frac{1}{p} |\nabla u^+|^p + \frac{1}{p'} \chi_{\{u>0\}} \, dx \right) \\
& \leq C(n,k) \left( \Big(\varepsilon \sqrt{k}+\frac{1}{p'} \Big) \int_{B_k} \chi_{\{u>0\}} \, dx + \frac{1}{p} \int_{B_k} |\nabla u^+|^p \, dx \right) \\
& \leq C_0 \int_{B_k \cap \{u>0\}} |\nabla u|^p+(p-1)\lambda_+^p \chi_{\{u>0\}} \, dx,
\end{aligned}
\end{equation}
where $p'$ is the conjugate of $p$ and
$$C_0:=C(n,k) \left( \varepsilon\sqrt{k}+ 1 \right).$$
Finally, putting together \eqref{Non-D7} and \eqref{Non-D8}, we reach to 
$$
\begin{aligned}
\int_{B_k(x_0) \cap \{u > 0 \}} & |\nabla u|^p+(p-1)\lambda_+^p \chi_{\{u>0\}}  \, dx \\
& \leq p(C\varepsilon)^{p-1} C_0 \int_{B_k(x_0) \cap \{u > 0 \}} |\nabla u|^p+(p-1)\lambda_+^p \chi_{\{u>0\}} \, dx,
\end{aligned}
$$
which implies that $u \equiv 0$ in $B_k(x_0)$ if $\varepsilon$ is small enough. This completes the proof of the non-degeneracy property.
\end{proof}

\medskip


\section*{Declarations}

\noindent {\bf  Data availability statement:} All data needed are contained in the manuscript.

\medskip
\noindent {\bf  Funding and/or Conflicts of interests/Competing interests:} The authors declare that there are no financial, competing or conflict of interests.

\addcontentsline{toc}{section}{\numberline{}References}

\bibliographystyle{acm}
\bibliography{mybibfile}

\end{document}